\documentclass{article}
\usepackage[T1]{fontenc}         
\usepackage{graphicx}
\usepackage{caption}
\usepackage{subcaption} 
\usepackage{float}   
\usepackage{fancybox}		  
\usepackage{makeidx} 
\usepackage{verbatim}
\usepackage{listings} 
\usepackage{fancyvrb}
\usepackage{amsfonts}
\usepackage{color}
\usepackage{colortbl}
\usepackage{amsmath}
\usepackage{color}
\usepackage{authblk}
\usepackage{epsfig}
\usepackage{verbatim}
\usepackage{listings}

\usepackage{amssymb}
\usepackage{amsbsy}
\usepackage{amsmath}
\usepackage{enumerate}
\usepackage{stmaryrd}
\usepackage{mathtools}

\usepackage{fancyhdr}

\usepackage{amsthm}

\DeclarePairedDelimiter\floor{\lfloor}{\rfloor}
\DeclarePairedDelimiter\ceil{\lceil}{\rceil}

\newtheorem{theorem}{Theorem}[section]
\newtheorem{prop}[theorem]{Proposition}
\newtheorem{remark}[theorem]{Remark}
\newtheorem{hyp}{Hypothesis}
\newtheorem{lemma}[theorem]{Lemma}
\newtheorem{prob}{Problem}

\usepackage[margin=1in]{geometry}

\begin{document} 

\date{}
\title{\LARGE \bf Optimal Control of a Collective Migration Model}
\author[*]{Benedetto Piccoli}
\author[*]{Nastassia Pouradier Duteil}
\affil[*]{ Department of Mathematical Sciences, Rutgers University, Camden, NJ 08102, USA.
        {\tt\small piccoli@camden.rutgers.edu, nastassia.pouradierduteil@rutgers.edu}}%
\author[$\dag$]{Benjamin Scharf}
\affil[$\dag$]{ Technische Universit\"at M\"unchen, Fakult\"at Mathematik, Boltzmannstrasse 3 D-85748, Garching bei M\"unchen, Germany.
        {\tt\small scharf@ma.tum.de}}%


\maketitle

\begin{abstract}
\noindent Collective migration of animals in a cohesive group is rendered possible by a strategic distribution of tasks among members: some track the travel route, which is time and energy-consuming, while the others follow the group by interacting among themselves.  
In this paper, we study a social dynamics system modeling collective migration. We consider a group of agents able to align their velocities to a global \textit{target velocity}, or to follow the group via interaction with the other agents. The balance between these two attractive forces is our control for each agent, as we aim to drive the group to consensus at the target velocity. 
We show that the optimal control strategies in the case of final and integral costs consist of controlling the agents whose velocities are the furthest from the target one: these agents sense only the target velocity and become \textit{leaders}, while the uncontrolled ones sense only the group, and become \textit{followers}. 
Moreover, in the case of final cost, we prove an "Inactivation" principle: there exist initial conditions such that the optimal control strategy consists of letting the system evolve freely for an initial period of time, before acting with full control on the agent furthest from the target velocity. 

\end{abstract}

\section*{Introduction}

A fascinating feature of large groups 
is their \emph{self-organization} ability,
i.e.  the emergence from local interaction rules of certain global patterns.
For instance, animal groups such as schools of fish, flocks of birds or herds of mammals exhibit strong coordination in their movements, see e.g. \cite{Bellomo, Bellomo2, Camazine, Couzin, CouzinKrause, Niwa, Parrish, Parrish2, Romey, Toner, Viscek}. This collective behavior in animal groups also inspired applications to robotics (see \cite{Berman}), in which the aim is to coordinate autonomous vehicles \cite{Chuang, Jadbabaie, Leonard, Sugawara} and flight formations \cite{Perea, Sepulchre}. Other interests concern models in microbiology \cite{Horstmann, Horstmann2, Keller, Patlak, Perthame}, pedestrian and crowd motions \cite{Cristiani, Cristiani2} and financial markets \cite{Bae, During, Lasry}.
Such systems are usually referred to as social dynamics.
Examples of self-organization include clustering of the agents, alignment of velocities, or other kinds of equilibria, see \cite{Camazine, Tadmor,  Motsch, Niwa, Parrish, Parrish2, Toner}. This raises the question of understanding the mechanisms behind
the global pattern formation. 

A well-known model was proposed by F. Cucker and S. Smale (see \cite{CuckerSmale}) to describe the phenomenon of {\it consensus} in terms of alignment of velocities in a group on the move.
The Cucker-Smale model in formula is written as: 
\begin{equation} \label{Model_CS}
\begin{cases}
\dot{x_i} = v_i \\
\dot{v_i} = \cfrac{1}{N} \sum\limits_{j=1}^N \cfrac{v_j - v_i}{(1+\| x_j-x_i\|^2)^\beta}
\end{cases}\quad \text{ for } i\in\{1,...,N\},
\end{equation}
where $\beta>0$, and $x_i\in\mathbb{R}^d$ and $v_i\in\mathbb{R}^d$ are respectively the {\it state} and {\it velocity}.
This model was originally designed to describe the formation and evolution of language, and the variables $v_i$ can more generally represent opinions, preferences or invested capital. The system converges to consensus if $\beta\leq \frac12$, which corresponds to a strong interaction even between distant agents, see \cite{Caponigro2, Caponigro}. On the other hand, if $\beta>\frac12$, i.e. if the interaction is too weak, convergence to consensus only happens under certain conditions. More generally, the term $(1+\| x_j-x_i\|^2)^{-\beta}$ can be replaced by $a(\|x_j-x_i\|)$. Intuitively, it is natural to define $a$ as a non-increasing function, since proximity
often encourages interaction. On the other hand, it was proven that interactions modeled by non-decreasing functions $a$, called heterophilious, in fact enhance consensus (see \cite{Motsch}). When the system does not converge to a desired state, a natural question is to study the possibility of steering it via controls functions $u_i$, in which case the second equation of \eqref{Model_CS} becomes: $\dot{v_i} = \frac{1}{N} \sum_{j=1}^N a(\|x_j-x_i\|) (v_j - v_i)+u_i$ (see \cite{Caponigro2, Caponigro, Fornasier}).

In the \emph{collective migration} problem (see \cite{Leonard1}), not only do agents interact with one another to travel as a group, but they also gather clues from the environment guiding them towards a global \emph{target velocity}. In the case of migrating birds, for instance, this velocity can be sensed through a magnetic field, the direction of the sun, or environmental features. However, sensing the migration velocity is costly, both in used time and energy. A trade-off thus occurs between gathering this information, which ensures more precision, and following the group, which is less costly and saves time and energy for other tasks such as surveying for predators \cite{Dall, Guttal}. This problem also applies to the field of robotics, in which gathering information from the environment is done at the expense of communicating with other robots (or planes, drones, etc.) or performing other tasks, and to the field of economics when one aims to influence decisions of a group based on limited information. This trade-off naturally separates the group into \emph{leaders}, who gather information, and \emph{followers}, who only interact with the other agents (see \cite{Guttal}).  

We study a \emph{Collective Migration Model}, where the agents' dynamics is determined by two forces: 
the attraction towards a target velocity $V$ (which we assume can be sensed)
and the consensus dynamics as in the Cucker-Smale model.
More precisely, each agent's evolution is governed by a parameter $\alpha_i\in [0,1]$ which provides the balance between the two forces.
The system can be written as: 
\begin{equation}\label{Model_Migration}
\begin{cases}
\dot{x_i} = v_i \\
\dot{v_i} = \alpha_i (V-v_i) + (1-\alpha_i)\cfrac{1}{N} \sum\limits_{j=1}^N a(\|x_j-x_i\|) (v_j - v_i)
\end{cases}\quad \text{ for } i\in\{1,...,N\},
\end{equation}
\noindent where
$x_i\in\mathbb{R}^d$ and $v_i\in\mathbb{R}^d$ are the state and velocity,
$V\in\mathbb{R}^d$ is the target velocity,
and $\alpha_i \in [0,1]$ is the control, 
with the constraint $\sum_i \alpha_i \leq M$, $M>0$. 
In this paper, we choose to set $a\equiv 1$, so that the strength of interaction does not depend on the agents' positions. This is a reasonable hypothesis for instance if we consider groups of planes or drones that can communicate just as easily from great distances. 

While the Cucker-Smale model leads to alignment of all velocities to the average one (when there is consensus), the migration model tends to align all velocities to the preassigned {\it target velocity}. Our work focuses on finding optimal control strategies in order to achieve consensus to the target velocity, and in particular on selecting optimal {\it controlled leaders} among the agents when the control strength $M$ is small with respect to the size of the group. 
In order to do that,
we define the cost function $\tilde{\mathbb{V}} = \frac1N \sum_i \|v_i-V\|^2$, measuring the distance from consensus at the target velocity. We first show that, given any $M>0$, the strategy to decrease $\tilde{\mathbb{V}}$ instantaneously, with the constraint $\sum_i \alpha_i\leq M$, consists of distributing the control among the agents with the largest positive projections of velocities along $\bar{v}-V$ (where $\bar{v}$ is the mean velocity). In particular, if $\langle v_i, \bar{v}-V\rangle <0$, the agent $i$ is not controlled ($\alpha_i=0$).

We then study the optimal control strategy to minimize $\tilde{\mathbb{V}}$ at a fixed final time and first focus on the case of two agents, with control bounded by $M\in [0,2]$. The optimal control strategies depend on $M$ but, in all cases, we act with larger control on the agent with the largest projected velocity. Furthermore, if the final time is too short to bring the agents together, then there are initial conditions for which
at first the system must evolve with no control ($\alpha\equiv 0$).
We call this phenomenon "Inactivation", in line with the "Inactivation Principle" proven in \cite{Gauthier} in the context of arm movements. In this collaborative work with biologists, the authors prove that during fast arm movements, it is optimal to simultaneously inactivate both agonistic and antagonistic muscles for a short moment nearing the peak velocity.
We next generalize our results to any number of agents, but with the constraint $M\leq 1$. Then the optimal control strategy acts
with full strength on a sub-group of agents to bring them together.
Also in this case we observe "Inactivation", which
occurs when the initial average velocity $\bar{v}$ is very close to the target velocity $V$.
Indeed, driving the system
to $V$ requires both achieving consensus and moving the average
velocity towards $V$. If the average velocity is already close
to $V$, then we are left with inducing consensus which
happens naturally without control.
However, simulations show that Inactivation is rare and its performance gain is very minor
compared to a full-control strategy. 

Then we move on to examine integral costs $\int_0^T\tilde{\mathbb{V}}(t)dt$ and show that the optimal control strategy 
never exhibits Inactivation. More precisely,
we must use full control at all time splitting it evenly among the agents with the biggest projected velocity. 
Such a strategy is more restrictive than that with final cost,
since the controls are completely determined by initial conditions,
while previously we could use any strategy bringing agents
together at final time. 

The paper is organized as follows. In Section \ref{Sec:gen}, we define the cost functional and make general observations. In Section \ref{Sec:instantaneous}, we determine the strategy to decrease it instantaneously in time. Then, in Section \ref{Sec_opt}, we introduce the optimal control problem to minimize the cost function at a given final time. We solve it for the particular case of two agents (Section \ref{Sec:2agents}) before generalizing to any number of agents with a control bounded by $1$ (Section \ref{Sec:gencase}). Lastly we find optimal control strategies to minimize the integral cost (Section \ref{Sec:int_cost}).

\section{Cost function and general observations} \label{Sec:gen}

With no loss of generality, we set the target velocity $V$ to zero. Having simplified the interaction function $a$, 
system (\ref{Model_Migration}) reduces to:
\begin{equation}
\begin{cases}
\dot{x_i} = v_i \\
\dot{v_i} = - \alpha_i v_i + (1-\alpha_i)\cfrac{1}{N} \sum\limits_{j=1}^N  (v_j - v_i)
\end{cases} \quad i\in\{1,...,N\}.
\label{dynamics}
\end{equation}
We set a final time $T>0$. Then given $M>0$, we define the set of controls $\mathcal{U}_M$ as:
\begin{equation}\label{Um}
\mathcal{U}_M=\Big\{\alpha:[0,T]\rightarrow [0,1]^N \Big | \; \alpha \text{ measurable, s.t. for all } t, \;  \sum\limits_{i=1}^N\alpha_i(t) \leq M\Big\}.
\end{equation}
\subsection{Projection of the Dynamics}

Note that the dynamics (\ref{dynamics}) can be written in the more compact way:
\begin{equation} \label{dynamics2}
\begin{cases}
\dot{x_i} = v_i \\
\dot{v_i} = - v_i + (1-\alpha_i) \; \bar{v} ,
\end{cases}
\end{equation}
where $\bar v$ represents the mean velocity $\bar v = \frac1N \sum_i v_i$. The evolution of $\bar v$ is given by $\dot{\bar v}=-\frac1N (\sum_i\alpha_i) \bar v$, so the direction of $\bar v$ is an invariant of the dynamics. 
We begin by assuming that the initial average velocity is different from the target one:
\begin{hyp}
\label{hyp_first}
$\bar{v}(0)\neq 0$.
\end{hyp}
This first assumption is only made in order to render the problem interesting. Indeed, if $\bar v(0)=0$, i.e. if the mean velocity is already at the target velocity $V$, then according to the evolution $\dot{\bar v} = -(\sum_i \alpha_i)\bar v$, it would hold $\bar v(t) = 0$ for all $t\geq 0$. Then looking at Equation \eqref{dynamics2}, we notice that the system is not controllable and that each velocity decreases exponentially to zero.
We can then define the invariant unit vector   $e = \frac{\bar{v}}{\|\bar{v}\|}$.

Let $w_i = v_i - \langle v_i, e \rangle \; e$ be the projection of $v_i$ over $(\bar{v}^\perp)$. Then
\[
\dot{w}_i  = -v_i + (1-\alpha_i) \; \bar{v} - \langle -v_i + (1-\alpha_i) \; \bar{v} , e \rangle \; e = - w_i.
\]
Therefore the projection of $v_i$ over $(\bar{v}^\perp)$ decreases exponentially, independently of the controls $\alpha_i$.
Let us now define $\xi_i = \langle v_i, e \rangle$. Its evolution is given by: $\dot\xi_i = - \langle v_i, e \rangle + (1-\alpha_i) \langle \bar{v}, e \rangle = - \xi_i + (1-\alpha_i) \| \bar{v} \| =  - \xi_i + (1-\alpha_i) \bar\xi$.
In the following, we will only study the equations governing the evolution of the projected variables $\xi_i$:
\begin{equation}
 \text{For all } i\in \{1,...,N\}, \quad \dot\xi_i =  - \xi_i + (1-\alpha_i) \bar\xi ,
\label{scalar}
\end{equation}
where $\bar\xi=\frac1N \sum_j \xi_j$. 
This is a significant result: instead of studying a system evolving in $\mathbb{R}^{Nd}$, we consider a system in $\mathbb{R}^N$, thus greatly reducing the complexity of theoretical and numerical analyses.
Hereafter we shall make the following hypothesis:
\begin{hyp}
\label{hyp_main} 
$\xi_i(0)\geq \xi_{i+1}(0)$ for every $ i\in \{1,...,N-1\}$.
\end{hyp}
\noindent  
This assumption allows us to order the initial projected velocities without loss of generality.
\begin{prop}\ \\
Having made Hyp. \ref{hyp_first} and Hyp. \ref{hyp_main}, it holds $\bar{v}(t) \neq 0$ and $\bar{\xi}(t)> 0$ for all $t \in [0,T]$. \\
Furthermore, let $\tau\in[0,T]$. If $\xi_i(\tau)\geq 0$, then $\xi_i(t)\geq 0$ for all $t\in[\tau,T]$. If $\xi_i(\tau)> 0$, then $\xi_i(t)> 0$ for all $t\in[\tau,T]$.
\label{prop_positive}
\end{prop}

\begin{proof}
The proposition is mainly a consequence of Gronwall's inequality: It holds
\begin{equation}\label{xibarvbar}
{\bar \xi}=\frac1N \sum_j \langle v_j,\frac{\bar{v}}{\|\bar{v}\|}\rangle =\langle \bar{v},\frac{\bar{v}}{\|\bar{v}\|}\rangle = \|\bar{v}\|
\end{equation}
and
$$\dot{\bar{\xi}} = - \frac1N \left(\sum_{i=1} \alpha_i \right) \bar{\xi} \geq - \frac{M}{N} \bar{\xi}.
$$
Hence, if $\bar{v}(0) \neq 0$ and therefore $\bar{\xi}(0)> 0$, then $\bar{\xi}(t)\geq  e^{-Mt/N}\bar{\xi}(0)>0$ and thus $\bar{v}(t) \neq 0$ for all $t \in [0,T]$. 
Now notice that from \eqref{scalar} we can compute for all $t\in [\tau,T]$: $\xi_i(t)=e^{-(t-\tau)} (\xi_i(\tau) + \int_\tau^t (1-\alpha_i)(s)\bar{\xi}(s)e^{s-\tau} ds)$, so $\xi_i(t)\geq e^{-(t-\tau)} \xi_i(\tau)$, which proves the second part of the proposition.
\end{proof}

\subsection{Migration functional}

We introduce the functional
\begin{equation}
\tilde{\mathbb{V}} = \frac1N \sum\limits_{i=1}^N \|v_i-V\|^2 ,
\end{equation}
which measures the distance from consensus at the desired velocity $V$. Since we set $V=0$, $\tilde{\mathbb{V}}$ reduces to: $\tilde{\mathbb{V}} = \frac1N \sum_i \|v_i\|^2 $. 
In the new projected coordinates $\xi$, the migration functional can be written as:
$\tilde{\mathbb{V}} = \frac1N \sum_i (\|w_i\|^2 + \xi_i^2)$, where only the second term $\xi_i^2$ can be controlled. 
Hence, here onward we will only consider the controllable part of $\tilde{\mathbb{V}}$, which we denote $\mathbb{V}$: 
\begin{equation}
\mathbb{V} = \frac1N \sum\limits_{i=1}^N \xi_i^2.
\label{Vxi}
\end{equation}
Notice that $\mathbb{V}$ can be written as a sum of two terms:
\begin{equation}\label{Vdecomp}
\mathbb{V}=\bar\xi^2+\frac1N\sum\limits_{i=1}^N(\xi_i-\bar\xi)^2,
\end{equation}
which should be minimized simultaneously (where we remind that $\bar{\xi}=\frac{1}{N}\sum_i\xi_i$).
Minimizing $\bar\xi ^2$ (or $\bar\xi$, since according to Proposition \ref{prop_positive}, $\bar\xi>0$) corresponds to steering the system as a whole to the desired velocity $V=0$. On the other hand, minimizing $ \frac1N\sum_i(\xi_i-\bar\xi)^2$ corresponds to driving the system to consensus. However, the dynamics (\ref{scalar}) of $\xi_i$ show that if $\xi_i<0$, decreasing $\bar\xi$ slows down the increase of $\xi_i$, resulting in a possible increase of $(\xi_i-\bar\xi)^2$. 
Hence, minimizing $\mathbb{V}$ requires balancing the decrease of the two terms in (\ref{Vdecomp}). 

\subsection{Minimization problems}\label{Sec:techniques}

In the following sections, we will deal with the minimization of different quantities, in order to design a strategy for consensus at the migration velocity $V=0$. Having fixed the final time $T$ a priori, we address three problems:
\begin{description}
  \item[($i$)] The minimization of $\frac{d \mathbb{V}}{dt}$, i.e. the maximization of the instantaneous decrease of $\mathbb{V}$ (see Section \ref{Sec:instantaneous}).
  \item[($ii$)] The minimization of the final cost $\mathbb{V}(T)$ (see Sections \ref{Sec_opt}, \ref{Sec:2agents} and \ref{Sec:gencase}).
  \item[($iii$)] The minimization of the integral cost $\int_0^T \mathbb{V}(t) dt$ (see Section \ref{Sec:int_cost}).
\end{description}

In order to minimize $(ii)$ $\mathbb{V}(T)$ and $(iii)$ $\int_0^T \mathbb{V}(t) dt$, we will design an optimal control strategy using Pontryagin's maximum principle. The minimization of $\dot{\mathbb{V}}$, on the other hand, will not provide an optimal control.

\section{Instantaneous Decrease}\label{Sec:instantaneous}
In this section we look for a control strategy maximizing the instantaneous decrease of $\mathbb{V}$. Strategies designed in this way are not optimal (in general), but are easier to study and can give a first good insight on the problem. Indeed, we will later compare the instantaneous decrease strategy to the optimal control strategies developed in Sections \ref{Sec:gencase} and \ref{Sec:int_cost}.

The time derivative of the migration functional $\mathbb{V}$ is given by: 
\[
\dot{\mathbb{V}} = \frac2N \sum\limits_{i=1}^N \xi_i \dot{\xi}_i = \frac2N \left( \sum\limits_{i=1}^N - \xi_i^2 + \sum\limits_{i=1}^N (1-\alpha_i) \bar{\xi} \xi_i  \right) 
 = -2 \mathbb{V} + \frac2N \bar{\xi}  \sum\limits_{i=1}^N (1-\alpha_i)\xi_i .
 \]
Since $\bar{\xi}\geq 0$, minimizing $\dot{\mathbb{V}}$ amounts to the following problem:
\[
\text{Find} \; \; \;  \text{min} \sum\limits_{i=1}^N (1-\alpha_i) \xi_i,
\]
which can be done as follows (where $\floor*{M}$ and $\ceil{M}$ respectively denote the floor and the ceiling of $M$, and $|\cdot|$ denotes the cardinality of a set): 
\begin{prop}\label{prop_inst}
Suppose that $\xi_1(t)\geq...\geq \xi_{N}(t)$ (or re-arrange the agents so that this is satisfied).
Then the following strategy minimizes $\frac{d}{dt}\mathbb{V}$ at time $t$: \\
Define $I^+(t)=\{i\in\{1,..,N\}, \; \xi_i(t) > 0 \}$. \\
If $|I^+(t)|\leq M$, then set $\alpha_i(t)=1$ if $i\in I^+$ and $\alpha_i(t)=0$ otherwise.\\
If $|I^+(t)|>M$ and $\xi_{\ceil{M-1}}>\xi_{\ceil{M}}>\xi_{\ceil{M+1}}$ then set
 $\alpha_i(t)=1$ if $i\leq \floor{M}$, $\alpha_{\floor{M}+1}(t)=M-\floor{M}$ and $\alpha_i(t)=0$ otherwise.\\
If $|I^+(t)|>M$ and $\xi_{\ceil{M-1}}=\xi_{\ceil{M}}$ or $\xi_{\ceil{M}}=\xi_{\ceil{M+1}}$, let $I_{\ceil{M}}=\{i\in\{1,...,N\}, \; \xi_i(t)=\xi_{\ceil{M}}(t)\}$ and $I_{\ceil{M}}^*=\{1, ...,\ceil{M}\}\setminus I_{\ceil{M}} $ . Then set 
$\alpha_i(t)=1$ if $i\in I_{\ceil{M}}^*$, $\alpha_i(t)=\frac{M-|I_{\ceil{M}}^*|}{|I_{\ceil{M}}|}$ 
if $i\in I_{\ceil{M}}$ and $\alpha_i(t)=0$ otherwise.\\
\end{prop}

\section{Optimal control for final cost} \label{Sec_opt}

In this section, we focus on 
problem $(ii)$ (see Section \ref{Sec:techniques}), i.e. 
minimizing the migration functional $\mathbb{V}$ at final time $T$ using Pontryagin's maximum principle.

Let us compute the Hamiltonian $H$ of the scalar system (\ref{scalar}): 
\begin{equation} \label{Ham}
H 
 = \sum\limits_{i=1}^N \lambda_i \left( -\xi_i + (1-\alpha_i) \bar{\xi} \right) = -  \bar{\xi} \sum\limits_{i=1}^N \alpha_i \lambda_i  + \sum\limits_{i=1}^N \lambda_i \left( - \xi_i + \bar{\xi} \right).
\end{equation}
By Pontryagin's maximum principle \cite{Pont}, if $\alpha\in\mathcal{U}_M$, associated with the trajectory $\xi$, is optimal on $[0,T]$, then there exists $\lambda : [0,T]\rightarrow \mathbb{R}^N$ such that $\dot{\xi}=\frac{\partial H}{\partial \lambda}$ and $\dot{\lambda}=-\frac{\partial H}{\partial \xi}$. Furthermore the following minimization condition holds for almost all $t\in [0,T]$:
\begin{equation} \label{MinH}
H(t,\xi(t),\lambda(t),\alpha(t))=\min\limits_{\beta\in\mathcal{U}_M} H(t,\xi(t),\lambda(t),\beta(t)).   
\end{equation}
Since $\bar{\xi}\geq 0$, minimizing $H$ requires to set $\alpha_i = 1$ on the biggest positive $\lambda_i$. 
The differential equation for the covectors $\lambda_i$ gives:
\begin{equation}
\dot{\lambda}_i = - \frac{\partial H}{\partial \xi_i} = \frac{1}{N}\sum\limits_{j=1}^N \alpha_j \lambda_j- \bar{\lambda} + \lambda_i, \quad  i \in \{1,...,N\}. \; 
\label{lambda}
\end{equation}
From this we can also compute the evolution of $\bar{\lambda}=\frac{1}{N}\sum_i\lambda_i$:
\begin{equation}\label{lambdabar}
\dot{\bar{\lambda}}=\frac{1}{N}\sum_{j=1}^N\alpha_j\lambda_j.
\end{equation}
Since the final condition for $\xi$ is not fixed, the final condition for $\lambda$ at time $T$ gives:
\begin{equation}
\lambda(T) = \nabla \mathbb{V}(\xi(T)) = \left( \frac2N \xi_1(T), ... , \frac2N \xi_N(T)\right).
\label{lambda_final}
\end{equation}

\begin{prop}\label{prop_equality}
If $\bar t>0$, $i,j\in\{1,...,N\}$ , and $\lambda_i(\bar t)=  \lambda_j(\bar t)$, then 
$\lambda_i(t)=  \lambda_j(t)$ for all $t$. 
In this case, for a given control $\alpha$, any control $\tilde{\alpha}$ satisfying $\tilde{\alpha_i}+\tilde{\alpha_j} = \alpha_i +\alpha_j$ and $\tilde{\alpha_k} = \alpha_k$ for every $k \neq i,j$ gives the same evolution of $\lambda$. If the control $\alpha$ satisfies the Pontryagin Maximum Principle, then the control $\tilde{\alpha}$ also does.
\end{prop}

\begin{proof}
Assume that at time $\bar{t}$, $\lambda_i(\bar{t})= \lambda_j(\bar{t})$.
Let us define $z_{ij} = \lambda_i - \lambda_j.$
The evolution of $z_{ij}$ is given by:
$\dot{z}_{ij} = \dot{\lambda}_i - \dot{\lambda}_j = \lambda_i - \lambda_j = z_{ij}.$
Hence, $z_{ij}(t)=z_{ij}(\bar{t}) e^{t-\bar{t}}$, and if $z_{ij}(\bar{t})=0$, then for all $t$, $z_{ij}(t)=0$, i.e. $\lambda_i(t)=\lambda_j(t)$.
From this it follows that if $\alpha$ minimizes the Hamiltonian $H$, then any control $\tilde{\alpha}$ satisfying $\tilde{\alpha_i}+\tilde{\alpha_j} = \alpha_i +\alpha_j$ and $\tilde{\alpha_k} = \alpha_k$ also minimizes $H$, since one easily sees from (\ref{Ham}) that $H^\alpha=H^{\tilde{\alpha}}$ (where we denote by $H^\alpha$ the Hamiltonian obtained with the control function $\alpha$).
\end{proof}

Still assuming that the projected velocities are initially ordered (Hypothesis \ref{hyp_main}), the following lemma will allow us to further assume that they are ordered at all time. 

\begin{lemma}\label{lemma_orderxi}\ \\
There exists an optimal strategy satisfying the following: For all $ t\in [0,T]$,
\begin{equation} \label{orderxi}
 \text{If } i<j, \text{ then } \xi_i(t)\geq\xi_j(t). 
\end{equation} 
\end{lemma}

\begin{proof}
Consider an optimal control strategy $\alpha\in \mathcal{U}_M$. \\
Define $\tau=\sup\{t\in [0,T]; \; \exists \beta\in \mathcal{U}_M \text{ s.t. } \mathbb{V}_{\beta}(T)=\mathbb{V}_\alpha(T) \text{ and } \xi^\beta \text{ satisfies } (\ref{orderxi}) \text{ on } [0,t]\}$, where $\mathbb{V}_\beta$ and $\xi^\beta$ denote respectively the migration functional and the dynamics driven by the control $\beta$.
Let us prove by contradiction that $\tau=T$. Suppose that $\tau<T$.
Then there exist $i,j \in \{1,...,N\}$ with $i<j$ such that $\xi_i^\beta(\tau)=\xi^\beta_j(\tau)$ and $\xi^\beta_j(t)>\xi^\beta_i(t)$ on $]\tau, \tau+\delta]$ for some $\delta>0$.
Design a control strategy $\tilde{\beta}$ such that on $[\tau,T], \; \tilde{\beta}_i=\beta_j, \;  \tilde{\beta}_j=\beta_i $ and for every $ k\in \{1,...,N\}\setminus\{i,j\}, \; \tilde{\beta}_k=\beta_k$. 
Then for all $ t\in [\tau, T], \; \xi^{\tilde{\beta}}_i(t)=\xi^\beta_j(t), \;$
$ \text{ and } \xi^{\tilde{\beta}}_j(t)=\xi^\beta_i(t)$. 
So for all  $t\in[\tau, \tau+\delta], \; \xi^{\tilde{\beta}}_i(t)\geq\xi^{\tilde{\beta}}_j(t) \text{ and } \mathbb{V}^{\tilde{\beta}}(T)=\mathbb{V}^\beta(T)$. 
Proceeding likewise for every pair of indices $(m,n)$ satisfying $m<n$ and $\xi^\beta_m(t)<\xi^\beta_n(t)$ on $]\tau, \tau+\delta]$ we are able to design a control strategy $\tilde{\beta}$ satisfying (\ref{orderxi}) on $[0,\tau+\delta]$ and $\mathbb{V}^{\tilde{\beta}}(T)=\mathbb{V}^\alpha(T)$, which contradicts the definition of $\tau$.
In conclusion, $\tau=T$, i.e. for all $ t\in [0,T]$, for every $i, j\in \{1,...,N\}$, if $i<j$ then $\xi_i(t)\geq\xi_j(t)$.
\end{proof}
Hence, from here onward we shall assume that the variables $\xi_i$ are ordered at all time.

\begin{hyp} \label{hyp_orderxi}
If $i<j$, then $\xi_i(t)\geq\xi_j(t)$ for all $t\in [0,T]$.
\end{hyp}

\noindent From Hyp.\ref{hyp_orderxi} and the transversality condition (\ref{lambda_final}), we know that the covectors are ordered at final time, i.e. $\lambda_1(T)\geq ... \geq \lambda_N(T)$. From Prop. \ref{prop_equality}, we can generalize this for any time $t$:
\begin{equation}\label{Condopt}
\lambda_1(t)\geq ... \geq \lambda_N(t) \quad \text{ for all } t \in [0,T].
\end{equation}
The Pontryagin Maximum Principle allows us to state the following:

\begin{prop} \label{prop_optstrat}
The optimal strategy requires controlling the agents with the biggest positive covectors. Let $\alpha\in\mathcal{U}_M$ be an optimal strategy and $\lambda_i, \; i\in\{1,...,N\}$ the corresponding covectors. Define: 
\begin{equation} \label{setI}
I_\lambda(t):=\Big \{ i\in \{1,...,N\} \; \Big | \; \lambda_i(t)\geq 0 \Big\} \quad \text{  and  }\quad  I_\lambda^+(t):=\Big \{i\in \{1,...,N\} \; \Big | \; \lambda_i(t)>0 \Big \}.
\end{equation}
If the set $I_\lambda(t)$ is empty, then there is no control on any agent: $\alpha_i(t)= 0$ for every $i$. \\
If the set $I_\lambda^+(t)$ is not empty, then there exists $i\in I_\lambda^+(t)$ such that $\alpha_i(t)> 0$. Furthermore, $\sum_j \alpha_j \geq \min(|I_\lambda^+(t)|,M)$.
\end{prop}
\begin{proof}
According to Pontryagin's maximum principle (\ref{MinH}), if the control $\alpha$ is optimal, then it minimizes the Hamiltonian $H$ (\ref{Ham}) for almost all $t\in[0,T]$. The only controllable part of $H$ is $\tilde{H}=-\bar{\xi}\sum_i\alpha_i\lambda_i$. Minimizing $H$ requires controlling the largest positive $\lambda_i$ with the maximum strength allowed, while setting $\alpha_i=0$ if $\lambda_i<0$. If $\lambda_i=0$, Pontryagin's maximum principle gives no information on $\alpha_i$.
\end{proof} 

\noindent This leads to a trichotomy of cases. 
\begin{itemize} \itemsep0em
\item The biggest positive $\lambda_i$'s are always controlled with maximum control: $\sum_{i\in I_\lambda^+} \alpha_i=\min (|I_\lambda^+|, M )$. 
\item If for $i,j$, $\lambda_i$ and $\lambda_j$ coincide (at a certain time, which implies at all time) then $\alpha_i$ and $\alpha_j$ are under-determined. The PMP only requests that $\alpha_i + \alpha_j = c$ where $c$ is given by the strength of the control to be used on the two agents. 
\item The negative $\lambda_i$'s are never controlled: if $\lambda_i<0$, then $\alpha_i = 0$.
\end{itemize}

\begin{remark}
The existence of an optimal control for the problem described above is ensured by the convexity of the sets $F(t,\xi)=\{\left(\xi_i + (1-\alpha_i)\bar{\xi}\right)_{i=1...N}, \; \alpha\in [0,1]^N, \; \sum_i\alpha_i\leq M \}$ (see \cite{BressPic}).
\end{remark}

\section{Final cost with two agents}\label{Sec:2agents}

For a clearer understanding of the mechanisms taking place, we consider the simple case of two agents in $\mathbb{R}^d$. We consider the sets of controls $\mathcal{U}_M$, where $0<M\leq 2$. 
Thus, system (\ref{scalar}) becomes:

\begin{equation}
\begin{cases}
\dot{\xi}_1 = -\xi_1 + (1-\alpha_1) \; \bar{\xi} \\
\dot{\xi}_2 = -\xi_2 + (1-\alpha_2) \; \bar{\xi}.
\end{cases}
\label{scalar2}
\end{equation}
Computing the difference of the two projected variables will also prove useful:
\begin{equation}\label{diffxi1}
\dot{\xi}_1-\dot{\xi}_2= -(\xi_1-\xi_2)-(\alpha_1-\alpha_2)\bar{\xi}.
\end{equation}

\noindent Three different situations may arise, depending on the value of the constraint on the control. Indeed, two constraints are set: $\alpha_1+\alpha_2 \leq M$, and $0\leq\alpha_i \leq 1$ for $i=1,2$. We differentiate the cases (a) $0<M\leq 1$, (b) $1<M<2$ and (c) $M=2$.

\subsection{Pontryagin's Maximum Principal} \label{Sec:PMP}

Notice that the migration functional can be written as:
\begin{equation}
\mathbb{V} = \frac12 (\xi_1^2 + \xi_2^2) = \frac14 \left( (\xi_1+\xi_2)^2 + (\xi_1-\xi_2)^2 \right)=\bar{\xi}^2+\left(\frac{\xi_1-\xi_2}{2}\right)^2, 
\label{functional_V}
\end{equation}
once again emphasizing the necessary trade-off between two terms: the mean velocity $\bar{\xi}$ and the distance between the agents $|\xi_1-\xi_2 |$.
Computing the Hamiltonian of the system gives: 

\begin{equation}
H(t,\xi,\lambda,\alpha) = - \bar{\xi} \;(\alpha_1\lambda_1+\alpha_2\lambda_2) + \frac{\xi_2-\xi_1}{2} (\lambda_1-\lambda_2).
\end{equation}

In line with Hyp. \ref{hyp_main}, two cases are possible: $\xi_1(0)=\xi_2(0)$ or $\xi_1(0)>\xi_2(0)$. The following proposition deals with the first case. 

\begin{prop}\label{prop_control2equal}
If $\xi_1(0)=\xi_2(0)$, then a control strategy $\alpha$ is optimal if and only if it satisfies $\alpha_1+\alpha_2\equiv M$ and $\xi_1(T)=\xi_2(T)$.
\end{prop}
\begin{proof}
Consider the control given by $\tilde\alpha_1\equiv\tilde\alpha_2\equiv \frac{M}{2}$. It achieves $[\frac12(\xi_1(T)-\xi_2(T))]^2=0$ and ensures the maximal decrease of $\bar\xi^2$, thus is optimal for the minimization of $\mathbb{V}(T)$ \eqref{functional_V}.
Still from \eqref{functional_V}, a control $\alpha$ is optimal if and only if it achieves $[\frac12(\xi_1(T)-\xi_2(T))]^2=0$, which is equivalent to $\xi_1(T)=\xi_2(T)$, and ensures the maximal decrease of $\bar\xi^2$, which is equivalent to $\alpha_1+\alpha_2\equiv M$.
\end{proof}

Hence, the case $\xi_1(0)=\xi_2(0)$ is fully understood. In the following, we will deal with more complex cases by assuming:

\begin{hyp}
 $\xi_1(0)>\xi_2(0)$. 
\label{hyp_signs}
\end{hyp}

Before studying each case in detail, we give general considerations on the relation between the control $\alpha$ and $\lambda$:
\begin{itemize}
\item[(a)] If $M \leq 1$, minimizing $H$ (i.e. maximizing $\langle \lambda,\alpha \rangle$) gives (see Fig.\ref{fig:M1}): 
$(\alpha_1,\alpha_2) = (M,0)$ if $0<\lambda_2 < \lambda_1$; 
$(\alpha_1,\alpha_2) = (M/2,M/2)$ if $0<\lambda_2=\lambda_1$; 
$(\alpha_1,\alpha_2) = (M,0)$ if $\lambda_2 < 0 < \lambda_1$; 
$(\alpha_1,\alpha_2) = (0,0)$ if $\lambda_2 < 0$ and $\lambda_1 < 0$.
\item[(b)] If $1<  M< 2$, minimizing $H$ gives (see Fig.\ref{fig:M12}): 
$(\alpha_1,\alpha_2) = (1,M-1)$ if $0<\lambda_2 < \lambda_1$;
$(\alpha_1,\alpha_2) = (M/2,M/2)$ if $0<\lambda_2=\lambda_1$;
$(\alpha_1,\alpha_2) = (1,0)$ if $\lambda_2 < 0 < \lambda_1$;
$(\alpha_1,\alpha_2) = (0,0)$ if $\lambda_2 < 0$ and $\lambda_1 < 0$.
\item[(c)] If $M \geq 2$, minimizing $H$ gives (see Fig.\ref{fig:M2}): 
$(\alpha_1,\alpha_2) = (1,1)$ if $0<\lambda_2\leq \lambda_1$;
$(\alpha_1,\alpha_2) = (1,0)$ if $\lambda_2 < 0 < \lambda_1$;
$(\alpha_1,\alpha_2) = (0,0)$ if $\lambda_2 < 0$ and $\lambda_1 < 0$.
\end{itemize}

\noindent Notice that in all three cases,  if $\lambda_1=\lambda_2$, then the Pontryagin maximum principle does not give sufficient
information since any combination of $\alpha_1$ and $\alpha_2$ such that $\alpha_1 + \alpha_2 = M$ minimizes the scalar product $-\langle \lambda,\alpha \rangle $ (see Figure \ref{fig:M}).

\begin{figure}[h]
        \centering
        \begin{subfigure}[b]{0.3\textwidth}
                \includegraphics[trim=0cm 0.5cm 0cm 0cm, clip=true,scale=0.5]{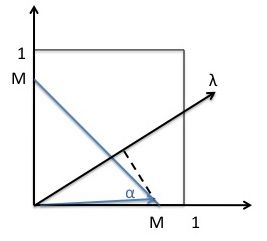}
                \caption{Case $M\leq 1$}
                \label{fig:M1}
        \end{subfigure}%
        ~ 
        \begin{subfigure}[b]{0.3\textwidth}
                \includegraphics[scale=0.4]{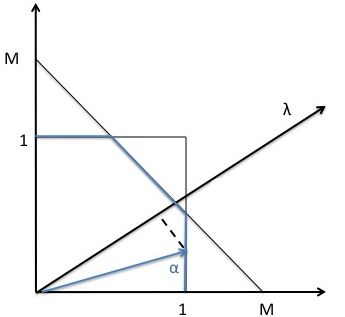}
                \caption{Case $1<M<2$}
                \label{fig:M12}
        \end{subfigure}
        ~ 
        \begin{subfigure}[b]{0.3\textwidth}
                \includegraphics[scale=0.3]{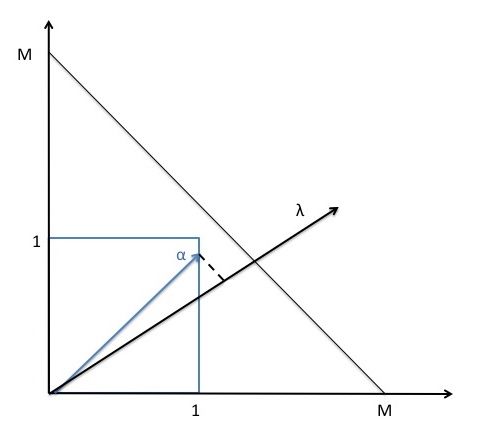}
                \caption{Case $M\geq 2$}
                \label{fig:M2}
        \end{subfigure}
        \caption{Minimizing $-\langle \lambda,\alpha \rangle $}\label{fig:M}
\end{figure}

\noindent The dynamics for $\lambda$ are given by $\dot{\lambda} = - \nabla H = \left(
    \begin{array}{c}
      \frac{1+\alpha_1}{2} \lambda_1 - \frac{1-\alpha_2}{2} \lambda_2 \\
       \frac{1+\alpha_2}{2} \lambda_2 - \frac{1-\alpha_1}{2} \lambda_1
    \end{array}
  \right) $,
which allows us to compute the evolution of the difference $\lambda_1-\lambda_2$: 
\begin{equation}
\frac{d}{dt}(\lambda_1-\lambda_2) = \lambda_1-\lambda_2.
\end{equation}
The transversality conditions give: $\lambda(T)=\nabla \mathbb{V}(T) = \left(\xi_1(T), \xi_2(T)\right)^T $.
Hence, if the final configuration is such that $\xi_1(T) \neq \xi_2(T)$, i.e. $\lambda_1(T) \neq
\lambda_2(T)$, the difference  $\lambda_1-\lambda_2$ increases with
time. On the other hand, if  $\lambda_1(T) = \lambda_2(T)$, then
$\forall t \leq T, \; \lambda_1(t) = \lambda_2(t)$.
If the dynamics allow us to drive $\xi_1$ and $\xi_2$ together before time
$T$, then $\lambda_1(t) = \lambda_2(t)$ for all $t$, and the Pontryagin maximum principle does not give sufficient information, as seen above.

\subsection{Global Strategy}

According to equation (\ref{functional_V}), the functional $\mathbb{V}$ can be written as:
\begin{equation}
\mathbb{V} = \bar{\xi}^2+  \frac{(\xi_1-\xi_2)^2}{4}.
\label{functional_V2}
\end{equation} 

\noindent Minimizing $\mathbb{V}$ requires minimizing $\bar{\xi}$ and $(\xi_1-\xi_2)^2$ simultaneously. The evolution of $\bar{\xi}$ is given by:
\begin{equation}
\dot{\bar{\xi}} = - \frac12  (\alpha_1+\alpha_2) \; \bar{\xi},
\label{xibar}
\end{equation}
while that of $(\xi_1-\xi_2)^2$ is:
\begin{equation}
\frac{d}{dt}\left( (\xi_1-\xi_2)^2 \right) = -2 (\xi_1-\xi_2)^2 -2 (\xi_1-\xi_2)\bar{\xi}(\alpha_1-\alpha_2).
\label{diffxi}
\end{equation}
Thus, minimizing $\bar{\xi}^2$ (both instantaneously and globally) requires using full control, i.e. setting $\alpha_1+\alpha_2=M$. 
On the other hand, the strategy to minimize $(\xi_1-\xi_2)^2$ is less clear. It would require both maximizing $\bar{\xi}$ and maximizing the difference $\alpha_1-\alpha_2$ (assuming that $\xi_1-\xi_2\geq 0$), and these conditions might not be compatible.

\subsection{Case $M=1$}


\begin{theorem}\ \\
\label{Th_M1}
Let $T>0$ and let $M=1$. Furthermore, let $\alpha=(\alpha_1,\alpha_2)\in\mathcal{U}_1$ (see \eqref{Um}) be an optimal control and $\xi$ be the corresponding trajectory of system \eqref{scalar2}. 
Define $t_0=2\ln(\xi_1(0) / \bar\xi(0))$. Then
\begin{itemize}
\item[(i)] $T\geq t_0$ if and only if $\xi_1(T)=\xi_2(T)$. In such a case, the control satisfies: $\alpha_1+\alpha_2\equiv 1$ (so $\bar{\xi}(t)=\bar{\xi}(0)e^{-t/2}$). For instance, the strategy $(\alpha_1,\alpha_2)(t)=(1,0)$ for all $t \in [0,t_0[$ and $(\alpha_1,\alpha_2)(t)=(1/2,1/2)$ for all $t \in [t_0,T]$ is optimal. 
\item[(ii)] If $T< t_0$, then $\alpha(t)=(0,0)$ for all  $t\in [0,t^*[$ and $\alpha(t)=(1,0)$ for all $t\in [t^*, T]$, where $t^*=2\ln(\bar X)$ and $\bar X\in [1,e^{T/2}[$ is defined as follows:
\end{itemize}
\begin{equation*}
\bar X=\arg\min_{X\in[1,e^{T/2}]}\left[ \left( \xi_1(0)+\bar{\xi}(0)(X^2-1) \right) ^2+\left( \xi_2(0)+\bar\xi(0)(X^2-1)+2\bar\xi(0)X(e^{T/2}-X)\right)^2 \right].
\end{equation*}
\end{theorem}

\begin{proof}
Let $\xi$ be an optimal trajectory achieved with optimal control $\alpha$.\\
To prove {\it(i)}, we shall show that the three statements $(a)$ $T\geq t_0$, 
$(b)$ there exists $t\in [0,T]$ such that $\xi_1(t)=\xi_2(t)$ and
$(c)$ $\xi_1(T)=\xi_2(T)$ are equivalent.\\
Suppose $(b)$
there exists $\tau\in [0,T]$ such that $\xi_1(\tau)=\xi_2(\tau)$. Then necessarily $\xi_1(T)=\xi_2(T)$. Indeed, suppose that $\xi_1(T)\neq \xi_2(T)$. Then any strategy $\tilde{\alpha}$ such that on $[0,\tau], \; \tilde{\alpha}=\alpha$ and on $]\tau, T], \; (\tilde{\alpha}_1,\tilde{\alpha}_2)=(\frac{\alpha_1+\alpha_2}{2},\frac{\alpha_1+\alpha_2}{2})$ achieves:  $\bar{\tilde{\xi}}(T)=\bar{\xi}(T)$ and $(\tilde{\xi}_1-\tilde{\xi}_2)^2(T)=0<(\xi_1-\xi_2)(T)$ (where $\tilde{\xi}$, $\tilde{V}$ denote the trajectory and cost corresponding to $\tilde{\alpha}$), so according to equation (\ref{functional_V2}), $\tilde{\mathbb{V}}(T)<\mathbb{V}(T)$ and control strategy $\alpha$ cannot be optimal. 
Hence, $\xi_1(T)=\xi_2(T)$.
\\Now suppose $(c)$ $\xi_1(T)=\xi_2(T)$. The transversality condition (\ref{lambda_final}) gives 
$\lambda_1(T)=\lambda_2(T)$ and from Proposition \ref{prop_equality} we get: $\lambda_1(t)=\lambda_2(t)$ for all $t\in [0,T]$. Then, $\dot{\bar{\lambda}}=\sum\alpha_i\lambda_i =(\sum\alpha_i)\bar{\lambda}$. Since $\bar{\xi}(T)>0$, the transversality condition (\ref{lambda_final}) gives: $\bar{\lambda}(T)>0$, and $\bar{\lambda}(t)=\lambda_1(t)=\lambda_2(t)>0$ for all $ t\in [0,T]$. Therefore, the set $I_\lambda$, see (\ref{setI}), is not empty, so according to Proposition \ref{prop_optstrat}, the optimal control strategy requires using maximal control strength: $\alpha_1+\alpha_2\equiv 1$. According to equation (\ref{xibar}), this suffices to fully determine $\bar{\xi}(t)=\bar{\xi}(0)\; e^{-t/2}$. 
Then $\xi_1(t)-\xi_2(t)=e^{-t}\left( (\xi_1-\xi_2)(0)-\bar{\xi}(0)\int_0^t(\alpha_1-\alpha_2)e^{s/2}ds\right)$, and 
$\xi_1(t)-\xi_2(t)=0$ if, and only if, $\int_0^t e^{s/2}(\alpha_1-\alpha_2)(s)ds = (\xi_1(0)-\xi_2(0)) / \bar{\xi}(0)$. Notice that $\min_{(\alpha_1,\alpha_2)\in\mathcal{U}_1}\{t \; |\; (\xi_1-\xi_2)(t)=0\}$ is obtained when $\alpha_1-\alpha_2$ is maximal, i.e. for $(\alpha_1,\alpha_2)\equiv(1,0)$. With this strategy, $\min_{(\alpha_1,\alpha_2)\in\mathcal{U}_1}\{t \; |\; (\xi_1-\xi_2)(t)=0\}:=t_0=2\ln(\xi_1(0) / \bar\xi(0))$. Hence, we must have: $T\geq t_0$. \\
Lastly, suppose $(a)$ $T\geq t_0$. Design a strategy $\tilde{\alpha}$ so that for all $t<t_0, \; (\tilde{\alpha}_1, \tilde{\alpha}_2)=(1,0)$ and for all $t \geq t_0, \; (\tilde{\alpha}_1, \tilde{\alpha}_2)=(1/2,1/2)$. This strategy is optimal since it maximizes the decrease of $\bar{\tilde{\xi}}$, see \eqref{xibar}, and achieves $(\tilde\xi_1-\tilde\xi_2)(T)=0$, see \eqref{diffxi}. Hence, our optimal strategy $\alpha$ must also satisfy: $\xi_1(T)=\xi_2(T)$ and $\alpha_1+\alpha_2\equiv 1$. This proves $(b)$. \\
We showed that $(a)$, $(b)$ and $(c)$ are equivalent. We thus proved the first part of the proposition: $T\geq t_0$ if and only if $\xi_1(T)=\xi_2(T)$. In this case, it also holds: $\alpha_1+\alpha_2\equiv 1$. \\
If on the other hand, {\it (ii)} $T<t_0$, then $\xi_1(t)>\xi_2(t)$ for all $t\in [0,T]$ (since $(b)$ implies $(a)$). According to condition (\ref{lambda_final}) and to Prop. \ref{prop_equality}, $\lambda_1(t)>\lambda_2(t)$ for all $t\in [0,T]$ and $\lambda_1(T)>0$. The evolution of $\lambda_1$ is given by: $\dot{\lambda}_1=\frac{1}{2}(\alpha_1\lambda_1+\alpha_2\lambda_2)+\lambda_1-\bar{\lambda}>0$ since $\lambda_1>\bar{\lambda}$. Hence, two cases must be distinguished: either $\lambda_1> 0$ at all time, so the set $I_\lambda^+$ is non-empty and full control will be used at all time, or there exists $ t^*\in]0,T[$ such that $\lambda_1< 0$ on $[0,t^*[$, $\lambda_1(t^*)=0$ and $\lambda_1> 0$ on $]t^*,T]$, in which case $\alpha=(0,0)$ on $[0,t^*[$ and $\alpha=(1,0)$ on $]t^*,T]$.
Knowing this, it is easy to express $\xi_1$, $\xi_2$ and $\mathbb{V}$ as functions of $t^*$:
\begin{equation}
\forall t\in [t^*,T], 
\begin{cases}
\xi_1(t)=e^{-t}(\xi_1(0)+\bar\xi(0)(e^{t^*}-1)) \\
\xi_2(t)=e^{-t}(\xi_2(0)+\bar\xi(0)(e^{t^*}-1)+2\bar\xi(0)e^{t^*/2}(e^{t/2}-e^{t^*/2}))\\
\mathbb{V}(t)=\xi_1^2(t)+\xi_2^2(t)
\end{cases}.
\end{equation}
Denoting $X=e^{t^*/2}$, $\mathbb{V}(T)$ can be written as a biquadratic polynomial in $X$:
\begin{equation*}
\mathbb{V}(T)(X)=e^{-2T}\left[ \left( \xi_1(0)+\bar{\xi}(0)(X^2-1) \right) ^2+\left( \xi_2(0)+\bar\xi(0)(X^2-1)+2\bar\xi(0)X(e^{T/2}-X)\right)^2 \right].
\end{equation*} 
We look for $\bar{X}$ minimizing $\mathbb{V}(T)(X)$ in the interval $[1,e^{T/2}]$ (so that $t^*\in[0,T]$).
Notice that the leading term is $2 e^{-2T} \bar{\xi}(0)^2 \cdot X^4$. Hence, there are at most two local minima in the interval $[1,e^{T/2}]$. 
Furthermore, $\mathbb{V}(T)(1)=e^{-2T}\left[\xi_1(0)^2+(\xi_2(0)+2\bar{\xi}(0)(e^{T/2}-1))^2\right]$ and $\mathbb{V}(T)(e^{T/2}) = e^{-2T}[ (\xi_1(0)+\bar{\xi}(0) (e^{T/2}-1))^2+ (\xi_2(0) + \bar{\xi}(0) (e^{T}-1))^2 ]$, so $\mathbb{V}(T)(1)<\mathbb{V}(T)(e^{T/2})$, which means that 
$\bar{X}< e^{T/2}.$ If $\bar{X} =1$, then $t^*=0$ so it is optimal to act with control $(1,0)$ on the full interval $[0,T]$. If $1<\bar{X}<e^{T/2}$, then $0<t^*<T$. The optimal control strategy will require leaving the system to evolve without control on $[0,t^*[$, and acting with control $\alpha=(1,0)$ on $[t^*,T]$. 
\end{proof}

\begin{remark}
The existence of an initial "Inactivation" period can be proven also with any number of agents (see Theorem \ref{th_inactivation}).
Numerical simulations with any number of agents (see Section \ref{Sec_gen_prac}) show that in some cases it is indeed optimal to let the system evolve without control on an initial time interval $[0, t^*]$, where $t^*>0$. 
\end{remark}

\subsection{Case $M<1$}

Generalizing to the case of any $M<1$, we conduct the same analysis and the optimal control strategy is similar.

\begin{theorem}\label{Th_Ms1}
Let T>0 and M<1. Let $\alpha=(\alpha_1,\alpha_2)\in\mathcal{U}_M$ (see \eqref{Um}) be an optimal control and $\xi$ be the corresponding trajectory of system \eqref{scalar2}.  Define $t_0=\frac{2}{2-M}\ln\left(\frac{2-M}{2M}(\xi_1(0)-\xi_2(0))/ \bar{\xi}(0)+1\right)$. Then
\begin{itemize}
\item[(i)] $T\geq t_0$ if and only of $\xi_1(T)=\xi_2(T)$. In this case, the control satisfies: $\alpha_1+\alpha_2\equiv M$ (so $\bar{\xi}(t)=\bar{\xi}(0) e^{-M t/2 }$).
\item[(ii)] If $T< t_0$, then there exists $t^* \in [0,T[$ such that $\alpha(t)=(0,0)$ for all  $t\in [0,t^*[$ and $\alpha(t)=(1,0)$ for all $t\in [t^*, T]$.

\end{itemize}
\end{theorem}
\begin{remark}
To compute $t^* \in [0,T[$ in the case $T< t_0$, one can compute $\mathbb{V}(T)(e^{t^*/2})$ depending on $t^* \in [0,T]$ similarly to the case $M=1$.
\end{remark}

\begin{proof}
Let $\xi$ be an optimal trajectory achieved with optimal control $\alpha \in \mathcal{U}_M$. We argue as in the case $M=1$.\\
To prove {\it (i)}, first suppose that there exists $\tau\in [0,T]$ such that $\xi_1(\tau)=\xi_2(\tau)$. Then, as in the case $M=1$, necessarily it holds $\xi_1(T)=\xi_2(T)$ and any strategy achieving $\xi_1(T)=\xi_2(T)$ while using maximum control $\alpha_1+\alpha_2\equiv M$ is optimal. Then $\xi_1(t)-\xi_2(t)=0 \Leftrightarrow \int_0^t e^{\frac{2-M}{2}s}(\alpha_1-\alpha_2)(s)ds = (\xi_1(0)-\xi_2(0))/ \bar{\xi}(0)$. Hence, $\min_{\alpha\in\mathcal{U}_M}\{t \; |\; (\xi_1-\xi_2)(t)=0\}$ is obtained when $\alpha_1-\alpha_2$ is maximal, i.e. for $(\alpha_1,\alpha_2)\equiv(M,0)$. With this strategy, $\min_{\alpha\in\mathcal{U}_M}\{t \; |\; (\xi_1-\xi_2)(t)=0\}=t_0$ as defined above. Hence, if there exists $\tau \leq T$ such that $\xi_1(\tau)=\xi_2(\tau)$, then $T\geq t_0$. This proves the first implication of the proposition: if $\xi_1(T)=\xi_2(T)$, then $T\geq t_0$.

Conversely, if $T\geq t_0$, then the strategy $(\tilde{\alpha}_1, \tilde{\alpha}_2)=(M,0)$ on $[0, t_0[$ and $(\tilde{\alpha}_1, \tilde{\alpha}_2)=(\frac{M}{2},\frac{M}{2})$ on $[t_0,T]$ is optimal since it minimizes $\bar{\tilde\xi}(T)$ and achieves $\tilde\xi_1(T)=\tilde\xi_2(T)$. Hence, if $\alpha$ is optimal, it must also satisfy $\alpha_1+\alpha_2\equiv M$ and $\xi_1(T)=\xi_2(T)$, which proves the second implication.

Now assume {\it (ii)} $T<t_0$. From {\it (i)} we get: $\xi_1(t)>\xi_2(t)$ for all $t \in [0,T]$. 
One can then argue as in the case $M=1$. According to Pontryagin's Maximum Principle, $\alpha_2\equiv 0$ and two cases have to be distinguished: either $\lambda_1>0$ at all time, so the set $I_\lambda^+$ (see \eqref{setI}) is non-empty and full control will be used at all time, or there exists $ t^*\in]0,T[$ such that $\lambda_1< 0$ on $[0,t^*[$, $\lambda_1(t^*)=0$ and $\lambda_1> 0$ on $]t^*,T]$, in which case $\alpha=(0,0)$ on $[0,t^*[$ and $\alpha=(M,0)$ on $]t^*,T]$.

\end{proof}

\begin{remark}
Notice that in the limit case $M\rightarrow 1$ of Theorem \ref{Th_Ms1}, one finds the same expression for $t_0$ as in Theorem $\ref{Th_M1}$.
\end{remark}

\subsection{Case $M=2$}

In order to determine the optimal strategy, let us first study the evolution of the covectors $\lambda$. From $\xi_1(T)\geq\bar\xi(T)>0$ (see Prop. \ref{prop_positive} and Hyp. \ref{hyp_orderxi}) and the transversality condition \eqref{lambda_final}, we get $\lambda_1(T)>0$. 

\begin{prop} \label{prop_M2}
Let $M=2$ and $\lambda_1$ and $\lambda_2$ be the covectors corresponding to an optimal control strategy for the system \eqref{scalar2}. Then they satisfy the following properties: 
\begin{itemize}
\item[(i)] If $\lambda_2(T)>0$, then $\lambda_1(t)>0$ and $\lambda_2(t)>0$ for all $t\in[0,T]$. 
\item[(ii)] If $\lambda_2(T)=0$, then $\lambda_1(t)>0$ and $\lambda_2(t)=0$ for all $t\in[0,T]$.
\item[(iii)] If $\lambda_2(T)<0$, then $\lambda_2(t)<0$ for all $t\in[0,T]$. 
\end{itemize}
\end{prop}

\begin{proof} \ \\
\it{(i)} Let $\lambda_2(T)>0$. Suppose that there exists $\tau\in [0,T[$ such that $\lambda_2(\tau)=0$ and $\lambda_2(t)>0$ for all $t\in ]\tau, T]$. Then since $\lambda_1\geq \lambda_2>0$ on $]\tau, T]$, according to Pontryagin's maximum principle (see Section \ref{Sec:PMP}), $(\alpha_1,\alpha_2)\equiv (1,1)$ on $]\tau, T]$, which gives the following evolutions: $\dot{\lambda}_1=\lambda_1$ and $\dot{\lambda}_2=\lambda_2$. Hence, $\lambda_2(\tau)=\lambda_2(T)e^{\tau-T}>0$, which contradicts the definition of $\tau$. Therefore, $\lambda_2(t)>0$ for all $t\in[0,T]$, and by \eqref{Condopt}, $\lambda_1(t)>0$ .\\
\it{(ii)} Let $\lambda_2(T)=0$. 
Let $\tau := \inf_{[0,T]}\{\bar t\in [0,T] \text{ s.t. } \lambda_2(t)=0 \text{ for all } t>\bar t\}$ and suppose that $\tau>0$. By definition of $\tau$, $\lambda_2(\tau)=0$. Since $\lambda_1(t)>\lambda_2(t)$ for all $t$ (see Prop. \ref{prop_equality}), there exists an interval $[\tau-\delta, \tau[$  on which $\lambda_1>0$ and either $\lambda_2>0$ or $\lambda_2<0$. If $\lambda_2(t)>0$ for all $t\in[\tau-\delta, \tau[$, then according to Pontryagin's maximum principle (Section \ref{Sec:PMP}), the control satisfies $\alpha_1(t)=\alpha_2(t)=1$, which gives: $\dot{\lambda}_2(t)=\lambda_2(t)>0$. So $\lambda_2(\tau)>0$, which contradicts the definition of $\tau$.
If on the other hand $\lambda_2(t)<0$ for all $t\in [\tau-\delta, \tau[$, then $\alpha_2(t)=0$ and $\dot{\lambda}_2(t)=\frac{1}{2}\lambda_2(t)<0$, which is impossible since it implies $\lambda_2(\tau)<0$. Hence, $\tau=0$. 
Furthermore, since $\lambda_1(T)>0$ and $\lambda_2\equiv 0$, then $\dot{\lambda}_1=\lambda_1$ in a neighborhood of $T$, which ensures that $\lambda_1(t)>0$ for all $t\in[0,T]$ (by the same reasoning as in \it{(i)}).\\
\it{(iii)} Let $\lambda_2(T)<0$. Define $\tau := \inf_{[0,T]}\{\bar t\in [0,T] \text{ s.t. } \lambda_1(t)>0 \text{ and }\lambda_2(t)<0 \text{ for all } t>\bar t\}$. Then on $]\tau,T]$, as seen in Section \ref{Sec:PMP}, $\alpha_1\equiv 1$ and $\alpha_2\equiv 0$, which gives: $\lambda_2(t)=\lambda_2(\tau)e^{T-\tau}$. Since $\lambda_2(T)<0$, it follows that $\lambda_2(\tau)<0$. Hence, either $\tau=0$ or $\lambda_1(\tau)=0$. 
Notice that since $\lambda_1(t)>\lambda_2(t)$ for all $t$, $\lambda_1$ is strictly increasing (see \eqref{lambda}). 
Then the former case implies that $\lambda_2(t)<0$ for all $t\in[0,T]$. In the latter case, we get that $\lambda_2(t)<0$ for all $t\leq \tau$.
\end{proof}

This information about the covectors allows us to solve the optimization problem based on the initial conditions and the final time.
 Recall 
 from Proposition \ref{prop_positive} 
 that $\xi_1(0)>0$.

\begin{theorem} \label{Th_M2}
Let M=2. Let $(\alpha_1,\alpha_2)\in\mathcal{U}_2$ be an optimal control strategy and $\xi$ be the corresponding trajectory for system \eqref{scalar2}. Define $t_0=2\ln\left(\xi_1(0)/(2\bar{\xi}(0))\right)$.
\begin{itemize}
\item[(i)] If $\xi_2(0)>0$, then $(\alpha_1,\alpha_2)\equiv (1,1).$
\item[(ii)] If $\xi_2(0)\leq 0$ and $T\geq t_0$, then $\xi_2(T)=0$ and $\alpha_1\equiv 1$. For instance the strategy $(\alpha_1,\alpha_2)=(1,0)$ for all $t \in [0,t_0[$ and $(\alpha_1,\alpha_2)=(1,1)$ for all $t \in [t_0,T]$ is optimal. 
Furthermore, if there exists $\bar{t} \in [0,T[$ such that $\xi_2(\bar{t})=0$, then $\xi_2(t)=0$ for all $t \in [\bar{t},T]$.
\item[(iii)] If $\xi_2(0)\leq 0$ and $T<t_0$, then there exists $t^* \in [0,T[$ such that $\alpha(t)=(0,0)$ for all $t\in [0,t^*[$ and $\alpha(t)=(1,0)$ for all $t\in [t^*, T]$.
\end{itemize}
\end{theorem}

\begin{proof}
Let $(\alpha_1,\alpha_2)$ be an optimal control strategy and $\xi$ be the corresponding trajectory. \\
{\it (i)} Let $\xi_2(0)>0$. 
According to Prop. \ref{prop_positive}, for all $t\in [0,T]$ it holds $\xi_1(t)>0$ and $\xi_2(t)>0$.  
Then $\lambda_1(T)>0$ and $\lambda_2(T)>0$. From Prop. \ref{prop_M2} it follows that  $\lambda_1(t)>0$ and $\lambda_2(t)>0$ for all $t\in [0,T]$. According to the PMP (see Section \ref{Sec:PMP}), maximal control has to be used at all time, i.e. $(\alpha_1,\alpha_2)(t)=(1,1)$ for all $t\in [0,T]$. 

For cases {\it (ii)} and {\it (iii)}, let $\xi_2(0)\leq 0$. 
By Prop.\ \ref{prop_positive} it holds $\xi_1(t)>0$ for all $t \in [0,T]$. Suppose that $\xi_2(T)>0$. Then from Prop.\ \ref{prop_M2} we get $\lambda_1(t)\geq \lambda_2(t)>0$ for all $t\in[0,T]$, so $(\alpha_1,\alpha_2)\equiv(1,1)$. But with this strategy $\dot \xi_2 = -\xi_2$, so $\xi_2(t)=\xi_2(0)e^{-t}\leq 0$ for all $t \in [0,T]$, which contradicts $\xi_2(T)>0$. Hence $\xi_2(T)\leq 0$. \\
{\it (ii)} First assume that $T\geq t_0$. 
Let us show that $\xi_2(T)=0$ and $\alpha_1\equiv 1$.
Such a strategy exists, since for instance the control $(\beta_1,\beta_2)(t)=(1,0)$ for $t \in [0,t_0[$ and $(\beta_1,\beta_2)(t)=(1,1)$ for $t \in [t_0,T]$ achieves $\xi_2^\beta(t)=0$ for all $t \in [t_0,T]$ (where $\xi^\beta$ denotes the trajectory corresponding to the control strategy $\beta$) -- by direct computation of \eqref{scalar2}. 
Suppose that $\xi_2(T)<0$. Then $\alpha$ cannot be optimal since the control strategy $\beta$
achieves the minimum of $\xi_1^\beta(T)^2$, see \eqref{scalar2}, and of $\xi_2^\beta(T)^2$ and therefore the minimum of $\mathbb{V}(T)=\xi_1^\beta(T)^2+\xi_2^\beta(T)^2$. Hence $\alpha$ must satisfy $\alpha_1\equiv 1$ and $\xi_2(T)=0$ in order to perform as well as $\beta$.  Obviously, all strategies that achieve $\xi_2(T)=0$ with $\alpha_1 \equiv 1$ achieve the same final positions (see \eqref{scalar2}) and thus have the same $\mathbb{V}(T)$. 
Furthermore, if there exists a $\hat{t}<T$ such that $\xi_2(\hat{t})=0$, then $\xi_2(t)=0$ for all $t \in [\hat{t},T]$: if $\xi_2(\bar t)=0$, then $\dot\xi_2(\bar t)=(1-\alpha_2)\bar\xi(\bar t) \geq 0$ and therefore $\xi_2$ cannot become negative, once it reaches $0$. On the other hand, if $\xi_2(t)>0$, then $\xi_2(T)>0$ by Prop.\ \ref{prop_positive}.\\
{\it (iii)} Assume now that $T< t_0$. Firstly, we show that an optimal strategy (by PMP) always achieves $\xi_2(T)<0$. We argue by contradiction: Assume that $\xi_2(T)=0$. 
Then $\lambda_1(T)>0$ and $\lambda_2(T)=0$ and, according to Proposition \ref{prop_M2}, it follows that $\lambda_1(t)>\lambda_2(t)=0$ for all  $t\in[0,T]$. According to the PMP, $\alpha_1 \equiv 1$.
 Then the growth of $\xi_2$ is maximal if, and only if, $\alpha_2 \equiv 0$ since in this case $\bar\xi$ is maximal. 
But with this strategy $\xi_2$ cannot reach $0$ before $t_0$ -- by direct computation of \eqref{scalar2}. Therefore $\xi_2(T)<0$, so $\lambda_2(T)<0$ and 
 $\lambda_2(t)<0$ for all $t \in [0,T]$ by Prop. \ref{prop_M2}. Hence we are in the same situation as in the case $M=1$ and $M<1$. Two cases are possible: either $\lambda_1> 0$ at all time, so the set $I_\lambda^+$ is non-empty and full control on $\xi_1$ is used at all time, or there exists $ t^*\in]0,T[$ such that $\lambda_1< 0$ on $[0,t^*[$, $\lambda_1(t^*)=0$ and $\lambda_1> 0$ on $]t^*,T]$, in which case $\alpha=(0,0)$ on $[0,t^*[$ and $\alpha=(1,0)$ on $]t^*,T]$.
\end{proof}

\begin{remark}
To compute $t^* \in [0,T[$ in the case $\xi_1(0)>-\xi_2(0)>0$ and $T<t_0$, one can compute $\mathbb{V}(T)(X)$ depending on $t^* \in [0,T[$ similarly to the case $M=1$.
\end{remark}

\subsection{Case $1<M<2$}

As in the case $M=2$, we state the following properties concerning the covectors $\lambda$. 

\begin{prop} \label{prop_M12}
Let $M\in ]1,2[$ and $\lambda_1$ and $\lambda_2$ be the covectors corresponding to an optimal control strategy for the system \eqref{scalar2}. They satisfy the following properties: 
\begin{itemize}
\item[(i)] If $\lambda_2(T)>0$, then $\lambda_1(t)>0$ and $\lambda_2(t)>0$ for all $t\in[0,T]$. 
\item[(ii)] If $\lambda_2(T)=0$, then $\lambda_1(t)>0$ and $\lambda_2(t)=0$ for all $t\in[0,T]$.
\item[(iii)] If $\lambda_2(T)<0$, then $\lambda_2(t)<0$ for all $t\in[0,T]$. 
\end{itemize}
\end{prop}

\begin{proof} \
The proof is very similar to that of Prop \ref{prop_M2}.\\
\it{(i)} Let $\lambda_2(T)>0$. Suppose that there exists $\tau\in [0,T[$ such that $\lambda_2(\tau)=0$ and $\lambda_2(t)>0$ for all $t\in ]\tau, T]$. Then if $\lambda_1> \lambda_2>0$ on $]\tau, T]$, according to Pontryagin's maximum principle (see Section \ref{Sec:PMP}), $(\alpha_1,\alpha_2)\equiv (1,M-1)$ on $]\tau, T]$, which gives $\dot{\lambda}_2=\frac{M}{2}\lambda_2$. If $\lambda_1= \lambda_2>0$ on $]\tau, T]$, then $\alpha_1+\alpha_2\equiv M$ (see Figure \ref{fig:M12}), which also gives $\dot{\lambda}_2=\frac{M}{2}\lambda_2$. Hence, $\lambda_2(\tau)=\lambda_2(T)e^{\frac{M}{2}(\tau-T)}>0$, which contradicts the definition of $\tau$. \\
For \it{(ii)} and \it{(iii)} we reason the same way as in the proof of Proposition \ref{prop_M2}.
\end{proof}

As in the previous sections, this allows us to solve the optimal control problem by distinguishing cases based on the initial conditions and the final time. The case $\xi_2(0)<0$ is illustrated in Figure \ref{fig:2agents}.

\begin{theorem}\label{Th_M12}\ \\
Let $M\in]1,2[$. Let $\alpha\in\mathcal{U}_M$ be an optimal control strategy and $\xi$ be the corresponding trajectory.\\
Define $t_0\leq t_1\leq t_2$ as: $t_0=2\ln\left(\frac{\xi_1(0)}{2\bar\xi(0)}\right), \; t_1=\frac{2}{2-M}\ln\left(\frac{\xi_1(0)}{2\bar\xi(0)}\right) \text{ and }  t_2=\frac{2}{2-M} \ln(\frac{\xi_1(0)}{\bar\xi(0)})$. \\
If $\xi_2(0) > 0$, two subcases are to be distinguished: 
\vspace{-2mm}
\begin{itemize}\itemsep0pt \parskip0pt \parsep0pt
\item If $T<t_2$, then $(\alpha_1,\alpha_2) \equiv (1,M-1)$ and $0<\xi_2(T)<\xi_1(T)$. 
\item If $T\geq t_2$, $\xi_1(T)=\xi_2(T)$ and $\alpha_1+\alpha_2=M$. 
\end{itemize}
\vspace{-2mm}
In the case $\xi_2(0)< 0$, four subcases appear: 
\vspace{-2mm}
\begin{itemize}\itemsep0pt \parskip0pt \parsep0pt
\item If $T<t_0$, then $\xi_2(t)<0$ and there exists $t^*\in [0,T[$ such that $(\alpha_1,\alpha_2)(t)= (0,0)$ for all $t\in [0,t^*]$ and $(\alpha_1,\alpha_2)(t)= (1,0)$ for all $t\in ]t^*,T]$.
\item If $t_0\leq T\leq t_1$, then $\alpha_1\equiv 1$ and $\xi_2(T)=0$.
\item If $t_1 < T<t_2$, then $(\alpha_1,\alpha_2)\equiv (1,M-1)$ and $0<\xi_2(T)<\xi_1(T)$. 
\item If $t_2\leq T$, then $\alpha_1+\alpha_2\equiv M$ and $\xi_1(T)=\xi_2(T)$. 
\end{itemize}
\end{theorem}

\begin{remark}
Notice that if $\xi_1(0)=\xi_2(0)$, then $t_2=0$.
\end{remark}

\begin{remark}
In the limit case $M\rightarrow 1$, the times $t_0$ and $t_1$ are equal, which is in line with Theorem \ref{Th_M1}. In the limit case $M\rightarrow 2$, $t_1$ and $t_2$ are undefined, in line with Theorem \ref{Th_M2}. 
\end{remark}

\begin{figure}[h]
        \begin{center}
        \begin{subfigure}[b]{0.2\textwidth}
                \includegraphics[trim=1.5cm 0cm 1.5cm 0.5cm, clip=true, scale=0.22]{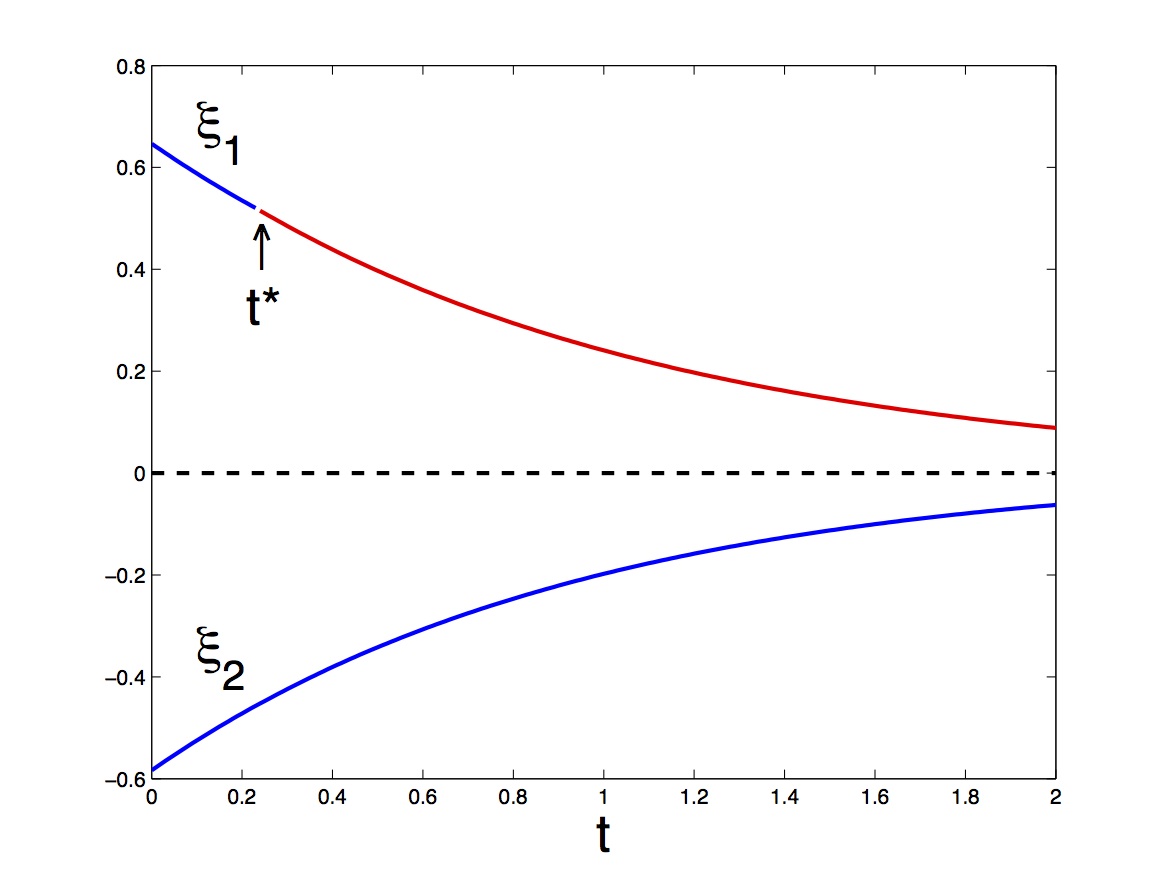}
                \caption{$T < t_0$}
                \label{fig:2agents1}
        \end{subfigure}%
        \hspace{10pt}
        \begin{subfigure}[b]{0.2\textwidth}
                \includegraphics[trim=1.5cm 0cm 1.5cm 0.5cm, clip=true, scale=0.22]{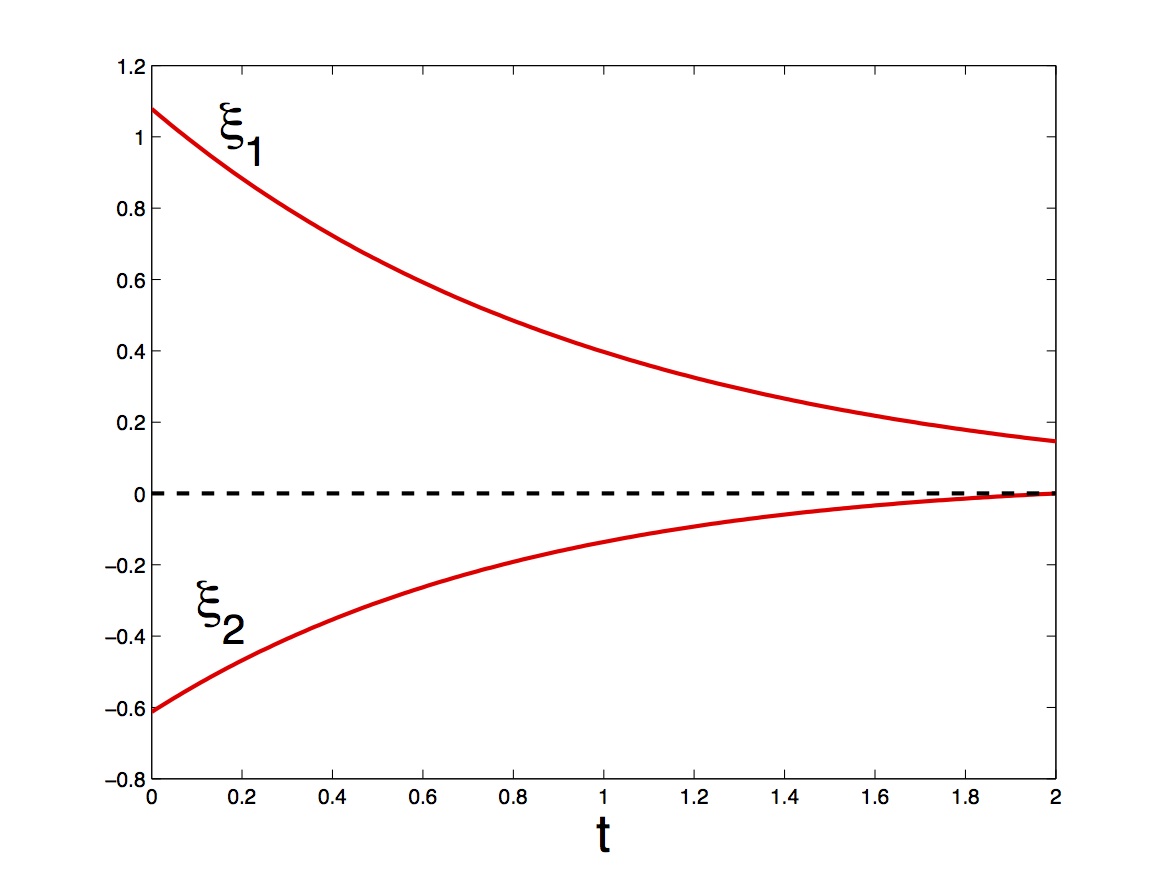}
                \caption{$t_0 \leq T\leq t_1$}
                \label{fig:2agents2}
        \end{subfigure}
        \hspace{10pt}
        \begin{subfigure}[b]{0.2\textwidth}
                \includegraphics[trim=1.5cm 0cm 1.5cm 0.5cm, clip=true, scale=0.22]{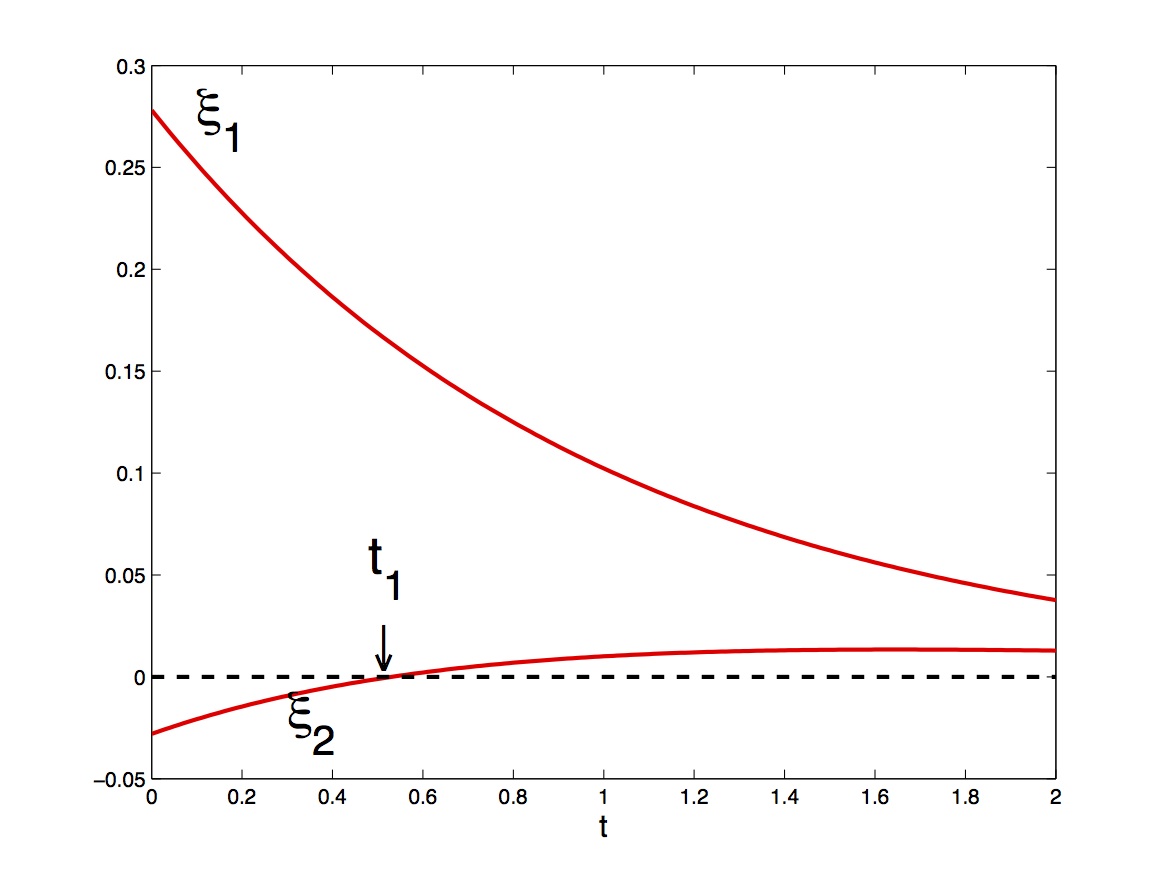}
                \caption{$t_1 < T <t_2$}
                \label{fig:2agents3}
        \end{subfigure}
        \hspace{10pt}
        \begin{subfigure}[b]{0.2\textwidth}
                \includegraphics[trim=1.5cm 0cm 1.5cm 0.5cm, clip=true, scale=0.22]{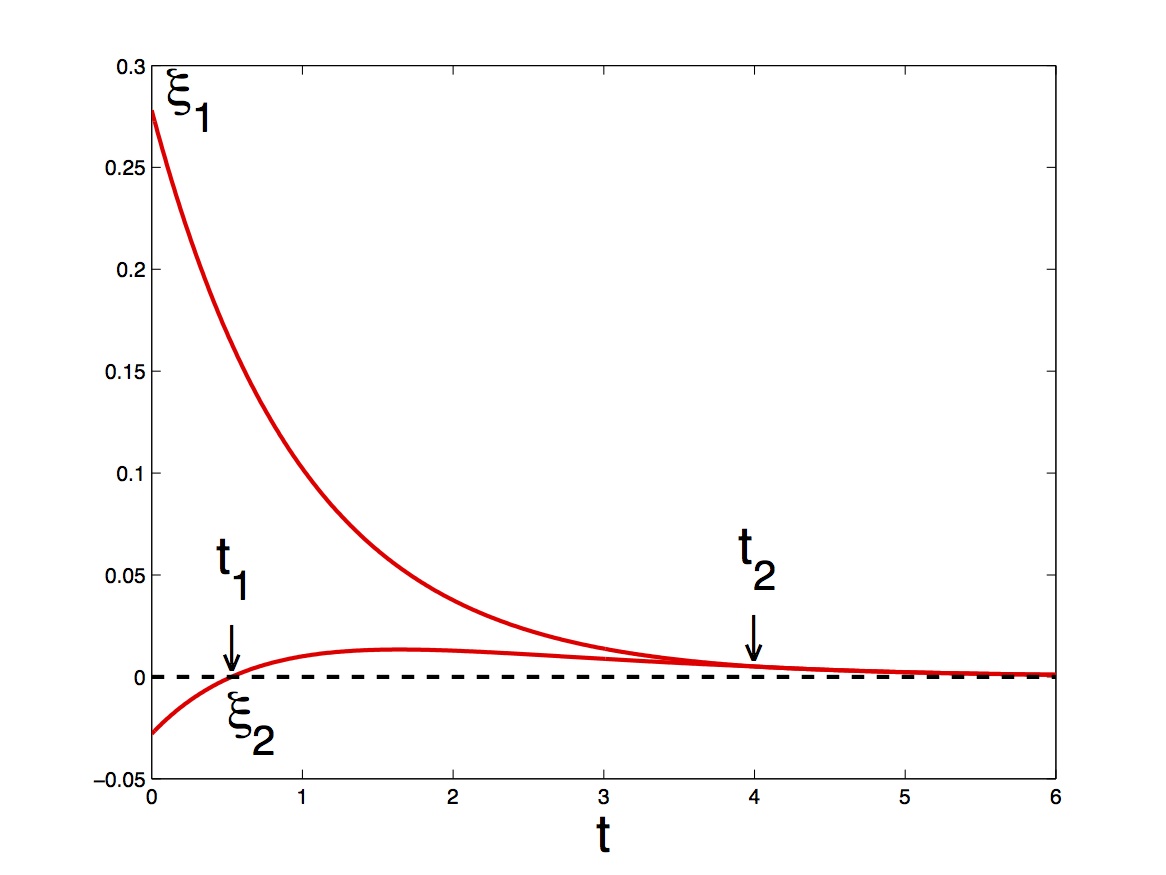}
                \caption{$t_2 \leq T$}
                \label{fig:2agents4}
        \end{subfigure}        
        \caption{Control strategies in the case $\xi_2(0)<0$ (controlled agents in red, uncontrolled ones in blue)}\label{fig:2agents}
        \end{center}
\end{figure}

\begin{proof}
See appendix. 
\end{proof}

\section{Final cost with any number of agents and control bounded by M=1} \label{Sec:gencase}

\subsection{Theroretical Analysis}

In this section, we address the optimal control problem of minimizing $\mathbb{V}(T)$ with any number of agents, setting the upper bound $M=1$ on the strength of the control, i.e. $\sum_{i=1}^N \alpha_i\leq 1$. We define the set of such controls:

\begin{equation} \label{setU}
\mathcal{U}=\Big\{\alpha:[0,T]\rightarrow [0,1]^N \Big | \; \alpha \text{ measurable, s.t. for all } t\in[0,T] \;  \sum\limits_{i=1}^N\alpha_i(t) \leq 1\Big\}.
\end{equation} \\
We
remind the equations governing the evolution of $\xi_i$ and $\bar{\xi}$  for $i\in\{1,...,N\}$:
\begin{equation}
\label{xidot}
\dot\xi_i=-\xi_i+(1-\alpha_i)\bar\xi \qquad \text{and} \qquad \dot{\bar{\xi}}=-(\sum_i\alpha_i) \; \bar{\xi}.
\end{equation}
As before, we aim to minimize the migration functional $\mathbb{V}=\frac1N\sum\limits_{i=1}^N\xi_i^2$ over the space $\mathcal{U}$ at final time:

\begin{prob}\label{prob_main}
Find $\arg\min\limits_{\alpha\in\mathcal{U}} \mathbb{V}(T).$
\end{prob}

Let us consider the restricted set of full-strength controls $\mathcal{U}_{FS}\subset\mathcal{U}$:
\begin{equation} \label{setUFS}
\mathcal{U}_{FS}=\Big\{\alpha:[0,T]\rightarrow [0,1]^N \Big | \; \alpha \text{ measurable, s.t. for all } t, \;  \sum\limits_{i=1}^N\alpha_i(t)= 1\Big\}.
\end{equation} 
We also introduce the set of optimal controls $\mathcal{U}_\text{opt}$:
\begin{equation} \label{setUopt}
\mathcal{U}_{\text{opt}}=\Big\{\alpha\in\mathcal{U}\; \text{ s.t. } \mathbb{V}_\alpha=\min\limits_{\beta\in\mathcal{U}}\mathbb{V}_\beta\Big\}.
\end{equation}
 A question then arises naturally: are there optimal controls among full-strength controls? In other words, we study the intersection $\mathcal{U}_\text{FS}\cap\mathcal{U}_\text{opt}$. To answer this, we first look for an optimal control strategy among the restricted set of controls $\mathcal{U}_\text{FS}$, i.e. we consider the problem:
\begin{prob} \label{prob_FS}
Find $\arg\min\limits_{\alpha\in\mathcal{U}_{FS}} \mathbb{V}(T).$
\end{prob}
Introducing the partial mean $\bar\xi_{1,l}=\frac1l \sum_{i=1}^l \xi_i$, we design the following optimal control strategy to solve Problem \ref{prob_FS}.

\begin{theorem}[Full-control strategy]\ \\ \label{th_strategy}
Let $T>0$. The strategy designed in Prop \ref{prop_inst} to decrease $\dot{\mathbb{V}}$ instantaneously is an optimal control strategy for Problem \ref{prob_FS}. It can be explicitly described as follows:\\
Define $t_1=0$ and for $l\in \{2,...,N\}, \; t_l = \frac{N}{N-1}\ln \left( (l-1)\frac{N-1}{N} \frac{\bar{\xi}_{1,l-1}(0)-\xi_{l}(0)}{\bar{\xi}(0)}+1 \right)$.\\
If there exists $l\in \{1,...,N-1\} \text{ such that } T\in [t_l, t_{l+1}[$, then any strategy satisfying: $\xi_i(T)=\bar\xi_{1,l}(T)$ for every $ i\in \{1,...,l\}$, $\sum_{i=1}^l\alpha_i\equiv 1$ and $\alpha_i\equiv 0$ for every $ i \in \{l+1,...,N\}$
 is optimal.\\
If $T\geq t_N$, then any strategy satisfying $\xi_i(T)=\bar{\xi}(T)$  for all $i\in \{1,...,N\}$ and $\sum_{i=1}^N\alpha_i\equiv 1$ is optimal.\\
For instance, if $T\in [t_l, t_{l+1}[$, one optimal strategy would consist in defining the following piecewise constant control:
\begin{equation}
 \forall k\leq l,\; \forall t\in [t_k,t_{k+1}[, \; 
\begin{cases}
\alpha_i(t)=\frac{1}{k}\; \text{if } i\leq k\\
\alpha_i(t)=0 \; \text{if } i>k.
\end{cases}
\end{equation}
\end{theorem}

\begin{proof}
Let us first show that if $T\geq t_l$, then the optimal control strategy for Problem \ref{prob_FS} must achieve $\xi_i(T)=\bar{\xi}_{1,l}(T)$ for all $i\in\{1,...,l\} $, reasoning by contradiction. \\
Suppose that there exists $k\in\{1,...,l\}$ such that $\xi_k(T)\neq \bar\xi_{1,l}(T)$. Using Hyp. \ref{hyp_orderxi}, we can suppose that there exists $m<l$ such that for every $i\in \{1,...,m\}, \; \xi_i(T)=\bar\xi_{1,m}(T)$, and for every $i\in \{1,...,m\}$ and $j\in \{m+1,...,N\}, \; \xi_j(T)<\xi_i(T)$.  

Let $j\in \{m+1,...,l\}$. The transversality condition (\ref{lambda_final}) gives:  for every $i\in \{1,...,m\}, \; \lambda_j(T)<\lambda_i(T)$. According to Proposition \ref{prop_equality}, for all $t\in [0,T]$, for every $i\in \{1,...,m\}, \; \lambda_j(t)<\lambda_i(t)$. According to the PMP, as seen in Section \ref{Sec_opt}, only the biggest covectors are controlled, and since $\alpha\in\mathcal{U}_\text{FS}$, with maximum control. So $\sum_{i=1}^m \alpha_i\equiv 1$ and $\alpha_j \equiv 0$.  
The evolutions of $\xi_j$ and $\bar\xi_{1,m}$ are then given by:
\begin{equation}
\begin{cases}
\dot{\xi}_j = -\xi_j+\bar\xi \\
\dot{\bar{\xi}}_{1,m} = -\bar{\xi}_{1,m} + \frac{m-1}{m}\bar{\xi}.
\end{cases}
\end{equation}
Since $\sum_{i=1}^N\alpha_i\equiv 1$, the evolution of the mean is given by $\dot{\bar\xi}=-\frac1N \bar\xi$, and we can compute $\bar\xi=\bar\xi(0) e^{-t/N}$, which in turn allows us to solve: 
\begin{equation}
\forall t\in [0,T], 
\begin{cases}
\xi_j(t)= e^{-t} \left( \xi_j(0) + \frac{N}{N-1}\bar\xi(0) (e^{\frac{N-1}{N}t}-1)\right) \\
 \bar\xi_{1,m}(t)= e^{-t} \left( \bar\xi_{1,m}(0) + \frac{m-1}{m}\frac{N}{N-1}\bar\xi(0) (e^{\frac{N-1}{N}t}-1)\right). \\
\end{cases}
\end{equation}
We get:
\begin{equation}
(\bar{\xi}_{1,m}-\xi_j)(T)= e^{-T} \left(\bar{\xi}_{1,m}(0)-\xi_j(0)-\frac1m \frac{N}{N-1}\bar{\xi}(0)(e^{\frac{N-1}{N}T}-1) \right).
\end{equation}
We made the hypothesis that $T\geq t_l = \frac{N}{N-1}\ln \left( (l-1)\frac{N-1}{N} \frac{\bar{\xi}_{1,l-1}(0)-\xi_{l}(0)}{\bar{\xi}(0)}+1 \right)$. Hence, 
\begin{equation}
\begin{split}
(\bar{\xi}_{1,m}-\xi_j)(T) & \leq e^{-T} \left(\bar{\xi}_{1,m}(0)-\xi_j(0)-\frac1m (l-1) (\bar{\xi}_{1,l-1}(0)-\xi_l(0)) \right) \\ 
& = \frac1m e^{-T} \left[ m \bar{\xi}_{1,m}(0) - m \xi_j(0)- (l-1)\bar{\xi}_{1,l-1}(0)+ (l-1)\xi_l(0) \right] \\
& \overset{(*)}{\leq} \frac1m e^{-T} \left[ \sum\limits_{i=1}^{m}\xi_i(0) - \sum\limits_{i=1}^{l-1}\xi_i(0) +(l-1-m)\xi_l(0) \right] \\
& = \frac1m e^{-T} \left[ - \sum\limits_{i=m+1}^{l-1} \xi_i(0) +(l-1-m)\xi_l(0) \right] \\
& \overset{(*)}{\leq} \frac1m e^{-T} \left[ -(l-1-m)\xi_l(0) +(l-1-m)\xi_l(0) \right] \\
& = 0,
\end{split}
\end{equation}
where inequalities $(*)$ derive from Hypothesis \ref{hyp_main}: since $j\leq l $, $\xi_j(0)\geq \xi_l(0)$. However, $(\bar{\xi}_{1,m}-\xi_j)(T)\leq 0$ contradicts that $\xi_j(T)<\xi_i(T)$ for every $i\in \{1,...,m\}$. 
From this we conclude that if $T\geq t_l$, then for every $ i\in \{1,...,l\}, \; \xi_i(T)=\bar{\xi}_{1,l}(T)$ for an optimal control strategy fulfilling Hypothesis \ref{hyp_orderxi}.

Let us now show that if $T<t_{l+1}$, then for every $k\in \{l+1,...,N\}, \; \alpha_i\equiv 0$ and $\xi_k(T)<\bar{\xi}_{1,l}(T)$.
\begin{equation}
\begin{split}
\bar{\xi}_{1,l}(T)-\xi_k(T) & \overset{(1)}{=} e^{-T} 
 \left( \bar{\xi}_{1,l}(0)-\xi_k(0)-\int_0^T e^{\frac{N-1}{N}s}(\frac{1}{l}\sum\limits_{j=1}^l \alpha_j - \alpha_k)(s)\bar{\xi}(0) ds \right) \\
 & \overset{(2)}{\geq} e^{-T} 
 \left( \bar{\xi}_{1,l}(0)-\xi_k(0)-\int_0^T e^{\frac{N-1}{N}s}\frac{1}{l}\bar{\xi}(0) ds \right) \\
 & = e^{-T} 
 \left( \bar{\xi}_{1,l}(0)-\xi_k(0)-\frac{N}{N-1}\frac{1}{l}\bar{\xi}(0) (e^{\frac{N-1}{N}T}-1) \right) \\
 & \overset{(3)}{>}   e^{-T} 
 \left( \bar{\xi}_{1,l}(0)-\xi_k(0)- (\bar{\xi}_{1,l}(0)-\xi_{l+1}(0) ) \right) \\
 & = e^{-T} 
 \left( \xi_{l+1}(0) -\xi_k(0) \right) \\
 & \overset{(4)}{\geq} 0, 
 \end{split}
\end{equation}
where: \\
(1) was computed using the evolutions of $\xi_k$ and $\bar{\xi}_{1,l}$: $\dot{\xi}_k=-\xi_k+(1-\alpha_k)\bar{\xi}$ and $\dot{\bar{\xi}}_{1,l} = -\bar{\xi}_{1,l} + (1-\frac{1}{l}\sum_{i=1}^l\alpha_i ) \bar{\xi}$,\\
(2) was obtained from inequalities $\sum_{j=1}^l \alpha_j(t)\leq 1$ and $\alpha_k(t)\geq 0$ for all $t$,\\
(3) comes from the inequality: $T<t_{l+1}=\frac{N}{N-1}\ln (\frac{N-1}{N}l\frac{\bar{\xi}_{1,l}(0)-\xi_{l+1}(0)}{\bar{\xi}(0)} +1)$, \\
(4) derives from Hypothesis \ref{hyp_main} since $k\geq l+1$. \\
Hence, for every $k\in \{l+1,...,N\}, \; \xi_k(T)\geq\bar{\xi}_{1,l}(T)$.
Furthermore, the transversality condition (\ref{lambda_final}) and Proposition \ref{prop_equality} imply that for all $t\in [0,T]$ for every $i\in \{1,...,l\}, \; \lambda_k(t)<\lambda_i(t)$ and the Pontryagin Maximum Principle as seen in Section \ref{Sec_opt} states that $\alpha_k\equiv 0$. So $\xi_k(T)\geq\bar{\xi}_{1,l}(T)$. 

We proved that if $T\in [t_l, t_{l+1}[$, then for every $i\in \{l+1,...,N\}$, $ \alpha_i\equiv 0$. Since $\bar{\xi}$ is fully determined as $\alpha\in\mathcal{U}_\text{FS}$, this means that for all $i\in \{l+1,...,N\}, \; \xi_i(T)$ is also fully determined (satisfying the equation $\dot{\xi}_i = -\xi_i + \bar\xi$ ). 
On the other hand, we proved that for all $i\in \{1,...,l\}, \; \xi_i(T)=\bar\xi_{1,l}(T)$ and that $\sum_{i=1}^l \alpha_i \equiv 1$, so $\bar{\xi}_{1,l}$ is also fully determined (satisfying the equation $\dot{\bar{\xi}}_{1,l}=-\bar{\xi}_{1,l}+\frac{l-1}{l}\bar{\xi}$ ).
Hence, any strategy such that for all $i\in \{1,...,l\}, \;\xi_i(T)=\xi_{1,l}(T)$ with $\sum_{i=1}^l\alpha_i\equiv 1$ and for all $i \in \{l+1,...,N\},\ \alpha_i\equiv 0$ is optimal for Problem \ref{prob_FS}.
\end{proof}

Notice that this optimal control strategy is not sparse, as control is split among more and more agents as time goes. However, it is not unique and one could very well act on one agent at a time until all reach the known final velocities.
Going back to the general Problem \ref{prob_main}, we prove that under certain conditions, the optimal control strategy uses full strength at all time, i.e. $\alpha^\text{opt}\in\mathcal{U}_{FS}$.

\begin{theorem}[Sufficient condition for full control]\ \\ \label{th_suffcond}
Define the time $t_N= \frac{N}{N-1}\ln\left(\frac{(N-1)^2}{N} \frac{\bar\xi_{1,N-1}(0)-\xi_N(0)}{\bar\xi(0)}+1\right)$ as in Theorem \ref{th_strategy}.\\
If $T\geq t_N$, then the optimal strategies $\alpha^\text{opt}$ to Problem \ref{prob_main}  belong to $\mathcal{U}_\text{FS}$ and for these controls $ \xi_i(T)=\bar{\xi}(T)$ for every $ i\in \{1,...,N\}$. 
  
\end{theorem}

\begin{proof}
If $T\geq t_N$, then the instantaneous decrease strategy designed in Theorem \ref{th_strategy} is optimal. Indeed, we noticed that the migration functional can be written as the sum of two terms (\ref{Vdecomp}): $\mathbb{V}=\bar\xi^2+\frac1N\sum(\xi_i-\bar\xi)^2$.
The strategy designed in Theorem \ref{th_strategy} minimizes $\bar{\xi}(T)$ by using full control at all time, hence minimizing $\bar{\xi}(T)^2$ since $\bar{\xi}>0$. Furthermore, it achieves $\xi_i(T)=\bar\xi(T)$ for all $i\in \{1,...,N\}$, thus minimizing the second term $\frac1N\sum(\xi_i-\bar\xi)^2$.
Hence any optimal control strategy has to use full control at all time and achieve $\xi_i(T)=\bar\xi(T)$ for every $i\in \{1,...,N\}$ in order to perform as well. 
\end{proof}

We finally address the general case stated in Problem \ref{prob_main}: minimize $\mathbb{V}(T)$ over the set of controls $\mathcal{U}$, for a given final time $T$. 
In the following theorem, we show the existence of an initial "Inactivation" time interval: the optimal strategy can require to let the system evolve freely (i.e. without control) at initial time, before acting on it with full strength. 

\begin{theorem}[Inactivation Principle] \label{th_inactivation}\ \\
If $T<t_N$, then one of the two holds: any control strategy $\alpha^{opt}$ either belongs to $\mathcal{U}_\text{FS}$ and the strategy designed in Theorem \ref{th_strategy} is optimal, or there exists some $\delta<T$ such that $\alpha^\text{opt}\equiv 0$ on $[0,\delta]$, and $\sum\alpha^\text{opt}_i \equiv 1$ on $[\delta, T]$.
\end{theorem}

\begin{proof}
According to Hypothesis \ref{hyp_orderxi}, we can assume that $\xi_1(T)\geq\xi_i(T)$ for every $i\in \{1,...,N\}$. Furthermore, $\bar\xi(T)>0$, so $\xi_1(T)>0$. From the transversality condition (\ref{lambda_final}) we deduce: $\lambda_1(T)\geq\lambda_i(T)$ for every $i\in \{1,...,N\}$ and $\lambda_1(T)>0$. From Prop. \ref{prop_equality}, we know that for all $t\in [0,T], \; \lambda_1(t)\geq\lambda_i(t)$. According to Prop. \ref{prop_optstrat}, 
full control is used at time $t$ if $\lambda_1(t) > 0$ and no control is used if $\lambda_1(t)<0$. 
Let us study the evolution of $\lambda_1$: $\dot{\lambda}_1=\frac1N\sum\alpha_j\lambda_j-\bar\lambda+\lambda_1$. By the Pontryagin maximum principle, we always have $\sum\alpha_j\lambda_j\geq 0$. 
Furthermore, $\lambda_1-\bar\lambda\geq 0$. So $\dot{\lambda}_1(t)\geq 0$ for all $t\in [0,T]$. 
We show that $\lambda_1=0$ at most at one point. Indeed, suppose that $\lambda_1(\tau)=0$ for some $\tau\in[0,T]$ and that $\dot{\lambda}_1(\tau)=0$.
 Then $\dot{\lambda}_1(\tau)=-\bar{\lambda}(\tau)$ so $\bar{\lambda}(\tau)=\lambda_1(\tau)=0$, and since the $\lambda_i$'s are ordered, $\lambda_i(\tau)=\bar{\lambda}(\tau)$ for every $i\in \{1,...,N\}$. According to Proposition \ref{prop_equality},  $\lambda_i(t)=\bar{\lambda}(t)$ for all time $t$ and every $i$. Since $\lambda_1(T)>0$, there exists a time interval $[\tau^*, T]$ such that $\lambda_1(t)>0$ for all $t\in[\tau^*,T]$. On this interval, $\dot{\lambda}_1=\frac{1}{N}\lambda_1\sum_j\alpha_j=\frac{1}{N} \lambda_1$, which gives: $\lambda_1(T)=\lambda_1(\tau^*)e^{\frac{1}{N}(T-\tau^*)}$. This contradicts the existence of a time $\tau$ at which $\lambda_1(\tau)=0$. In conclusion, if $\lambda_1(\tau)=0$, then $\dot{\lambda}_1(\tau)>0$ so $\lambda_1=0$ at most at one point. \\
Hence, there is a dichotomy of cases: \\
Either $\lambda_1(t)\geq 0$ for all time, so $I(t)\neq\emptyset$ for all $t$, which implies that $\alpha^\text{opt}\in\mathcal{U}_\text{FS}$ according to Prop. \ref{prop_optstrat}. In this case, $\arg\max_{\alpha
\in\mathcal{U}}\mathbb{V}=\arg\max_{\alpha
\in\mathcal{U}_\text{FS}}\mathbb{V}$ and the control strategy designed in Theorem \ref{th_strategy} for Problem \ref{prob_FS} is optimal also for Problem \ref{prob_main}. \\
 Or there exists $\delta\in[0,T]$ such that $\lambda_1(t)< 0$ on $[0,\delta[$ and $\lambda_1(t)\geq 0$ on $[\delta, T]$, which implies that $\alpha(t) \equiv 0 \text{ on } [0,\delta] \text{ and } \sum\alpha_i(t) \equiv 1 \text{ on } ]\delta, T]$. Practically, an optimal control strategy would consist in letting the system evolve without control on $[0,\delta[$. Then the full-control strategy from Theorem \ref{th_strategy} can be applied on $[\delta, T]$ with the new initial positions $\xi(\delta)$. 
\end{proof}

\begin{remark}
Although this result may seem counter-intuitive, in certain cases it makes sense to let the system evolve freely, at least initially. Indeed, without control the system naturally regroups in order to reach consensus, minimizing $\sum_{i=1}^N(\xi_i-\bar{\xi})$ in (\ref{Vdecomp}), but keeping $\bar{\xi}$ constant. Actual examples of such cases are shown in the next section. 
\end{remark}

\begin{remark}
Note that a constraint $M<1$ would not change the nature of the results. It would only mean acting with less strength on the controlled agents, therefore changing the values of the times $t_l$ defined in Theorem \ref{th_strategy}, but the optimal control strategy would be unchanged.
With a constraint $M>1$, we can expect results similar to those of Section \ref{Sec:2agents}, with two kinds of Inactivation periods, consisting either in letting the system evolve freely, or in controlling it with a non-maximal total strength $0<\sum_i\alpha_i<M$ (see Theorem \ref{Th_M2} (ii) and (iii)). 
\end{remark}

\subsection{Practical Approach} \label{Sec_gen_prac}

We proved in the previous section that the optimal strategy can either be to act with full control as in Theorem \ref{th_strategy}, or to let the system evolve without control on some time interval $[0,\delta]$, before acting with full control on $]\delta,T]$. In this section, we explore the practicality of Inactivation strategies.

 First, we run numerical simulations to find cases in which the optimal strategy involves Inactivation. The migration functional $\mathbb{V}_\delta$ can be computed explicitly as a function of $\delta$. We then look for the value of $\delta$ that minimizes $\mathbb{V}_\delta(T)$. 
Let us denote by $\xi^\delta$ the solution to system (\ref{xidot}) when no control is applied on $[0,\delta]$ and full control is used on $]\delta,T]$. Equation (\ref{xidot}) gives:
\begin{equation}
\begin{cases}
\dot{\xi}^\delta_i=-\xi^\delta_i+\bar\xi^\delta \\
\dot{\bar{\xi}}^\delta=0
\end{cases}
\quad \text{ on } [0,\delta], \;
\end{equation}
which allows us to solve:
$\xi^\delta_i(\delta)=e^{-\delta}\left(\xi^\delta_i(0)+\bar{\xi}^\delta(0)(e^{\delta}-1)\right)$.
We then apply the strategy designed in Theorem \ref{th_strategy} with the new initial conditions $\xi^\delta(\delta)$ and the new final time $T-\delta$. Define the times $t^\delta_1=0$ and for $l\in \{2,...,N\}, \; t^\delta_l = \frac{N}{N-1}\ln \left( (l-1)\frac{N-1}{N} \frac{\bar{\xi}^\delta_{1,l-1}(\delta)-\xi^\delta_{l}(\delta)}{\bar{\xi}^\delta(\delta)}+1 \right)$.
Find $l\in \{1,...,N-1\}, \text{ such that } T-\delta\in [t^\delta_l, t^\delta_{l+1}[$. Then any strategy satisfying $\xi^\delta_i(T)=\bar\xi^\delta_{1,l}(T)$ for every $ i\in \{1,...,l\}$, $\sum_{i=1}^l\alpha_i(t)= 1$ for all $ t\in [\delta,T]$,   and $\alpha_i\equiv 0$ for every $i \in \{l+1,...,N\}$ is optimal.
From equation (\ref{xidot}) we get: 
\begin{equation}
\begin{cases}
\dot{\bar{\xi}}^\delta_{1,l}=-\bar\xi^\delta_{1,l}+\frac{l-1}{l}\bar\xi^\delta \\
\dot{\xi}^\delta_{i}=-\xi^\delta_{i}+\bar\xi^\delta \; \; \text{ for } i\in \{l+1,...,N\}\\
\dot{\bar{\xi}}^\delta=-\frac{1}{N}\bar\xi^\delta 
\end{cases}
\quad \text{ on } [\delta,T],
\end{equation}
from which we can solve:
\begin{equation}
\begin{cases}
 \xi^\delta_i(T)=\bar\xi^\delta_{1,l}(T)=e^{-(T-\delta)}\left(\bar\xi^\delta_{1,l}(\delta)+\frac{l-1}{l}\frac{N}{N-1}\bar\xi^\delta(0)(e^{\frac{N-1}{N}(T-\delta)}-1)\right) \quad \text{ for all } i\in \{1,...,l\}, \\
\xi^\delta_i(T)=e^{-(T-\delta)}\left(\xi^\delta_i(\delta)+\frac{N}{N-1}\bar\xi^\delta(0)(e^{\frac{N-1}{N}(T-\delta)}-1)\right) \quad \text{ for all } i\in \{l+1,...,N\}.
\end{cases}
\end{equation}
We can now compute $\mathbb{V}^\delta(T)=\frac1N\sum\limits_{i=1}^N\xi_i^\delta(T)^2$ and numerically look for $\min\limits_{\delta\in[0,T]} \mathbb{V}^\delta(T)$ (see Figure \ref{fig:inactiontime}).

\begin{figure}[h!]
\begin{center}
\includegraphics[trim=0cm 0cm 0cm 0cm, clip=true, scale=0.5]{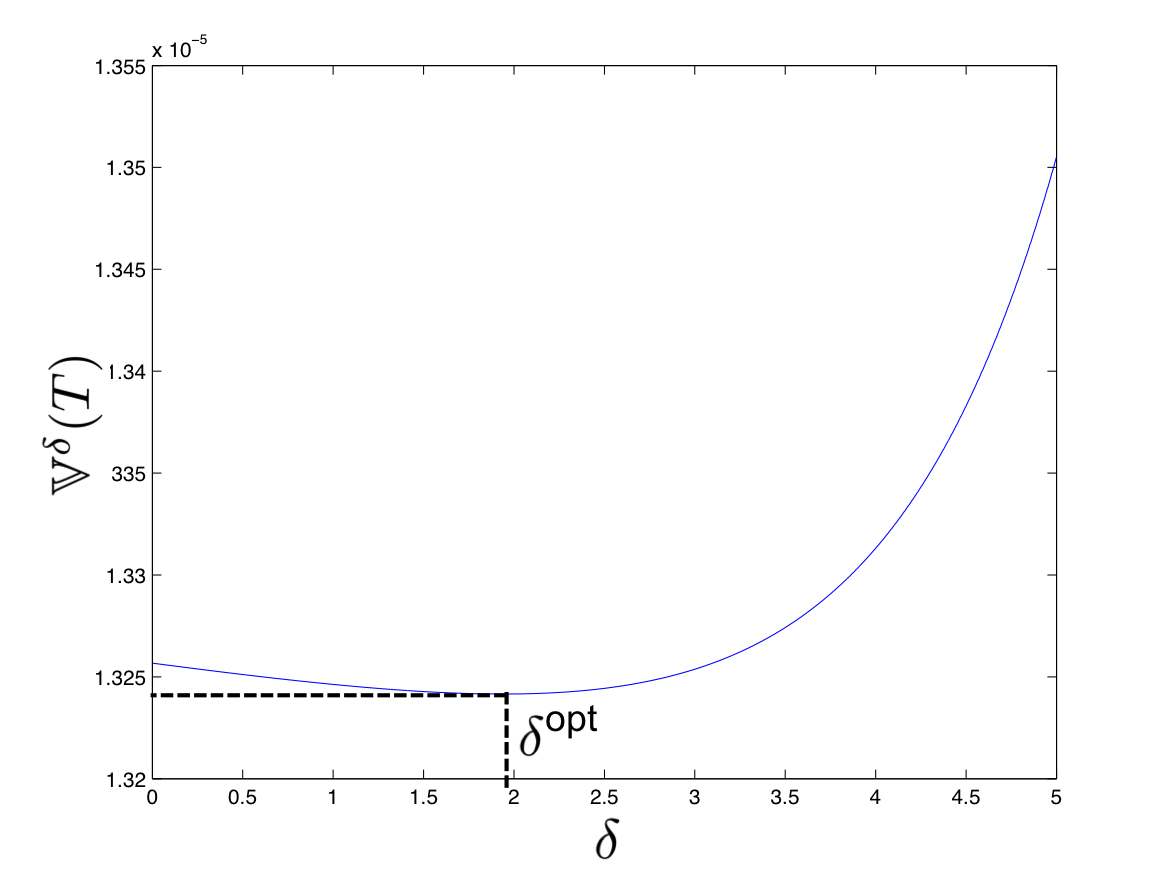}
\end{center}
\caption{$V^\delta(T)$ with respect to Inactivation time $\delta$. Here the optimal Inactivation time is $\delta=1.94$.}
\label{fig:inactiontime}
\end{figure}

Series of simulations were run to look for cases in which $\delta>0$. Table \ref{table_numcases} 
lists the percentage of such cases found over 1000 simulations, for different values of the number of agents and of the final time.
Initial projected variables $\xi_i(0)$ were chosen randomly in the interval $[-1,1]$ and such that the mean $\bar{\xi}$ is strictly positive. 
 As expected (and proven in Theorem \ref{th_suffcond}), for larger values of $T$, it is always optimal to act with full control at all time (in other words $\delta=0$). One also notices that as the number of agents increases, "Inactivation" cases become less and less frequent. 

\begin{table}[h!] 
\begin{center}
  \begin{tabular}{| c | c | c | c | c |}
    \hline 
    Number of agents & 5 & 10 & 20 & 50 \\ \hline \hline
    T=3 & 1.6 \%  & 0.9 \%  & 0 & 0 \\ \hline
    T=4 & 1.8 \%  & 0.7 \%  & 0.3 \%  & 0 \\ \hline
    T=5 & 1.0 \%  & 0.2 \%  & 0.2 \%  & 0 \\ \hline
    T=6 & 0.2 \%  & 0.1 \%  & 0 & 0.1 \%  \\ \hline
    T=7 & 0 & 0 & 0 & 0 \\ 
    \hline
  \end{tabular}
\end{center}
\caption{Percentage of cases in which $\delta>0$ out of 1000 simulations. $\xi_i(0)$ chosen randomly in $[-1, 1]$.}
\label{table_numcases}
\end{table}

  
Table \ref{table_V} shows the average of the relative difference $\frac{\mathbb{V}_\text{fc}-\mathbb{V}^\delta}{\mathbb{V}_\text{fc}}$, where $\mathbb{V}^\delta$ was obtained by using optimal control and $\mathbb{V}_{\text{fc}}$ by using full control at all time (as designed in Theorem \ref{th_strategy}). The gain in performance when using the optimal strategy is minor (significantly less than 1\% in most cases), and decreases as the number of agents increases. 

\begin{table}[h!]
\begin{center}
  \begin{tabular}{| c | c | c | c | c |}
    \hline 
    Number of agents & 5 & 10 & 20 & 50 \\ \hline \hline
    T=3 & 0.073\% & 0.001\% & - & - \\ \hline
    T=4 & 0.27\% & 0.018\% & 0.001\% & - \\ \hline
    T=5 & 0.91\% & 0.056\% & 0.0069\% & - \\ \hline
    T=6 & 1.53\% & 0.2\% & - & 0.00003 \% \\ 
    \hline
  \end{tabular}
\end{center} 
\caption{Average relative improvement of $\mathbb{V}^\delta$ w.r.t. $\mathbb{V}_{\text{fc}}$ }
\label{table_V}
\end{table}

The occurrence of Inactivation cases can be explained by looking at the two terms in the migration functional $\mathbb{V}=\bar\xi^2+\frac1N\sum(\xi_i-\bar\xi)^2$ (\ref{Vdecomp}). When $\bar{\xi}^2$ is small, the control strategy should concentrate on minimizing the second term $\frac1N\sum(\xi_i-\bar\xi)^2$, which does not necessarily require full control since the system naturally evolves to minimize this term. 
To confirm this reasoning, we look at the ratio $R:=(\frac1N\sum(\xi_i-\bar\xi)^2)/\bar{\xi}^2$ in one set of simulations ($N=5$, $T=3$) and find that the Inactivation cases correspond exactly to the largest values of $R$. Furthermore, the larger the ratio, the longer the Inactivation interval (see Figure \ref{fig:RatioInactivation}). 

 \begin{figure}[h!]
\begin{center}
\includegraphics[trim=0cm 0cm 0cm 0cm, clip=true, scale=0.6]{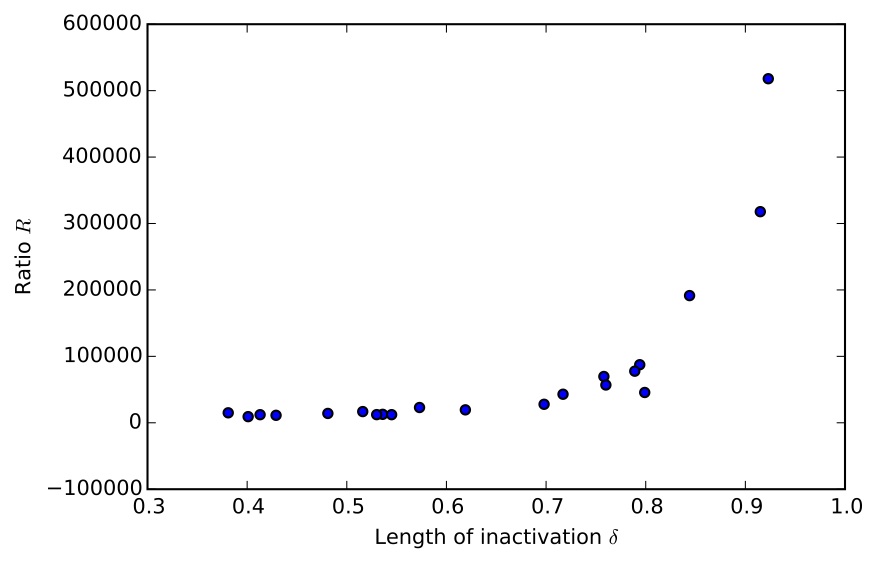}
\end{center}
\caption{Ratio $R:=(\frac1N\sum(\xi_i-\bar\xi)^2)/\bar{\xi}^2$ as a function of the length of the Inactivation interval $\delta$, for 20 simulations involving Inactivation with $N=5$ and $T=3$. The Inactivation $\delta$ increases as $\bar{\xi}^2$ tends to zero.}
\label{fig:RatioInactivation}
\end{figure}


Hence, $\mathcal{U}_\text{opt}\cap\mathcal{U}_\text{FS}=\emptyset$ occurs in very few cases, namely those in which $\bar\xi^2\ll\frac1N\sum(\xi_i-\bar\xi)^2$. Furthermore, when Inactivation exists, the gain in performance compared to the full control strategy is very minor. For reasons of computational speed and complexity, it is very reasonable to neglect those cases and to apply the full control strategy at all time. 

Figure \ref{fig:strategy} shows the evolution of the projected velocities $\xi_i, \; i\in \{1,...,10\}$ with respect to time, in a case where the optimal strategy requires full control at all time, with $T>t_{10}$. The control function is the one designed in Theorem \ref{th_strategy} and acts first on $\xi_1$, then on $\xi_1$ and $\xi_2$, and so on until all have reached consensus (in terms of the projected velocities $\xi_i$), at which point it acts with equal strength on all agents to drive $\bar\xi$ down to $0$. 

 \begin{figure}[h!]
\begin{center}
\includegraphics[trim=1cm 0.5cm 0cm 0cm, clip=true, scale=0.4]{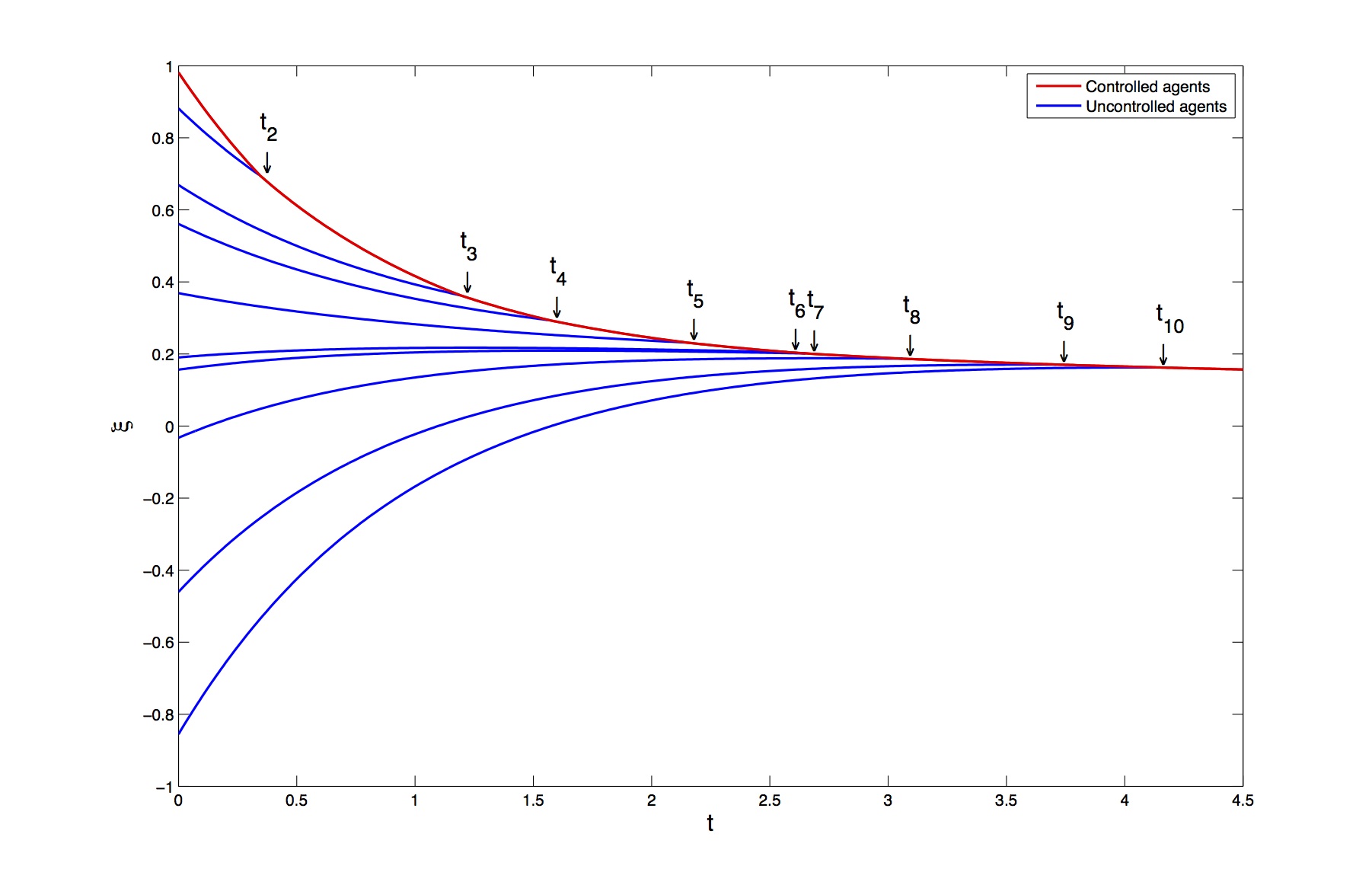}
\end{center}
\caption{Evolution of the projected velocities $\xi_i$ with the full strength optimal control for a system of $10$ agents. In this example $\bar{\xi}(0)=0.25$ so full control at all time is needed to drive $\bar{\xi}$ to the desired velocity $V=0$ (i.e. $\delta=0$). At final time $T=4.5$ the system has reached consensus, but not yet at the desired velocity.}
\label{fig:strategy}
\end{figure}

\section{Optimal control for integral cost}\label{Sec:int_cost}

In this section we focus on 
minimizing the integral of the migration functional, with the constraint on the controls $M=1$.
As done in Section \ref{Sec:gencase}, we define two problems (where $\mathcal{U}$ (\ref{setU}) and $\mathcal{U}_{\text{FS}}$ (\ref{setUFS}) are defined as before).
\begin{prob}\label{prob_main2}
Find $\arg\min\limits_{\alpha\in\mathcal{U}} \int_0^T \mathbb{V}(t) dt.$
\end{prob} 

\begin{prob}\label{prob_FS2}
Find $\arg\min\limits_{\alpha\in\mathcal{U}_{FS}} \int_0^T \mathbb{V}(t) dt.$
\end{prob}

\subsection{Pontryagin's Maximum Principle}\label{Sec:PMPint}

We first prove general results, with the aim of solving Problem \ref{prob_main2}. In order to use Pontryagin's maximum principle, we introduce the new Hamiltonian $H=\langle \lambda, f\rangle + \lambda^0\mathbb{V}$ 
and the equations governing the covectors' evolution $\dot\lambda_i=-\frac{\partial H}{\partial \xi_i}$. Considering normal trajectories, we set $\lambda^0=1$ and obtain: 
\begin{equation}\label{Hlambda_int}
\begin{cases}
H=\sum\limits_{i=1}^N(-\lambda_i\xi_i)+\bar\xi\sum\limits_{i=1}^N(1-\alpha_i)\lambda_i+\sum\limits_{i=1}^N\xi_i^2 \\
\dot\lambda_i=\lambda_i-\frac1N\sum_j (1-\alpha_j)\lambda_j - 2\xi_i .
\end{cases}
\end{equation} 
Since the final condition is not fixed, we have the following transversality condition for the covectors: 
\begin{equation}\label{lambda_final2}
\lambda(T)=0.
\end{equation}
As in the minimization of the migration functional at final time (Section \ref{Sec_opt}), we define $I_\lambda$ and $I_\lambda^+$ (see \eqref{setI}). Then minimizing $H=\sum_{i=1}^N-\alpha_i\lambda_i + \tilde{H}$ (where $\tilde{H	}$ contains only uncontrolled terms) requires the following : if $k \not\in I_\lambda$, $\alpha_k=0$; furthermore, if $I_\lambda^+\neq\emptyset$, then $\sum_{i\in I_\lambda^+} \alpha_i =1$.

As in Section \ref{Sec:gencase}, we make Hypothesis \ref{hyp_main}. Given the initial order on the agents' projected velocities $\xi_i$, we prove the following:

\begin{lemma}
There exists an optimal control strategy satisfying:
\begin{equation}\label{orderxi2}
\forall t\in [0,T], \; \forall i,j\in\{1,...,N\}, \; i<j \Rightarrow \xi_i(t)\geq\xi_j(t).
\end{equation}
\end{lemma}

\begin{proof}
The proof is very similar to that of Lemma \ref{lemma_orderxi}.
Consider an optimal control strategy $\alpha\in \mathcal{U}$. \\
Define $\tau=\sup\{t \; | \; \exists \beta\in \mathcal{U} \text{ s.t. } \int_0^T\mathbb{V}_{\beta}(s)ds=\int_0^T\mathbb{V}_\alpha(s)ds \text{ and } \xi^\beta \text{ satisfies } (\ref{orderxi2}) \text{ on } [0,t]\}$.
Let us prove by contradiction that $\tau=T$. Suppose that $\tau<T$.
Then there exist $i,j \in \{1,...,N\}$ with $i<j$ such that $\xi_i^\beta(\tau)=\xi^\beta_j(\tau)$ and $\xi^\beta_j(t)>\xi^\beta_i(t)$ on $]\tau, \tau+\delta]$ for some $\delta>0$.
Design a control strategy $\tilde{\beta}$ such that on $[\tau,T], \; \tilde{\beta}_i=\beta_j, \;  \tilde{\beta}_j=\beta_i $ and for every $ k\in \{1,...,N\}\setminus\{i,j\}, \; \tilde{\beta}_k=\beta_k$. 
Then for all $ t\in [\tau, T], \; \xi^{\tilde{\beta}}_i(t)=\xi^\beta_j(t), \;$
$ \text{ and } \xi^{\tilde{\beta}}_j(t)=\xi^\beta_i(t)$. 
So for all  $t\in[\tau, \tau+\delta], \; \xi^{\tilde{\beta}}_i(t)\geq\xi^{\tilde{\beta}}_j(t) \text{ and for all } t\in [0,T], \; \mathbb{V}^{\tilde{\beta}}(t)=\mathbb{V}^\beta(t)$.
Proceeding likewise for every pair of indices $(m,n)$ satisfying $m<n$ and $\xi^\beta_n(t)>\xi^\beta_n(t)$ on $]\tau, \tau+\delta]$ we are able to design a control strategy $\tilde{\beta}$ satisfying
(\ref{orderxi2}) on $[0,\tau+\delta]$ and $\int_0^T\mathbb{V}_{\beta}(t)dt=\int_0^T\mathbb{V}_\alpha(t)dt$, which contradicts the definition of $\tau$.
In conclusion, $\tau=T$, i.e. for all $ t\in [0,T]$, for every $i, j\in \{1,...,N\}, \; i<j \Rightarrow \xi_i(t)\geq\xi_j(t)$.
\end{proof}

Hence, as in Section \ref{Sec:gencase}, we can assume Hypothesis \ref{hyp_orderxi}: 
for all $t\in [0,T]$, if $i<j$, then $\xi_i(t)\geq\xi_j(t)$.
By the following proposition, we shall prove that the same order is observed among the covectors $\lambda_i$. 

\begin{prop}\label{prop_orderlambda2}
\begin{equation}
\forall t\in[0,T], \; i<j \Rightarrow \lambda_i(t)\geq\lambda_j(t).
\end{equation}
\end{prop}

\begin{proof}
Let us reason by contradiction. Suppose that there exists $\tau\in[0,T]$ such that for some i<j, $(\lambda_i-\lambda_j)(\tau)<0$. From the evolution of the covectors (\ref{Hlambda_int}) we derive for all $t\geq\tau$: $(\lambda_i-\lambda_j)(t)=e^{t-\tau}\left((\lambda_i-\lambda_j)(\tau)-2\int_\tau^t e^{-(s-\tau)}(\xi_i-\xi_j)(s)ds\right)$. Since $(\lambda_i-\lambda_j)(\tau)<0$ and for all $s\in[0,T], \; (\xi_i-\xi_j)(s)\geq 0$, we deduce that for all $t\in [\tau,T]$, $(\lambda_i-\lambda_j)(t)<0$, which contradicts the final condition (\ref{lambda_final2}).
\end{proof}

\begin{prop}\label{prop_difflambda0}
Let $\tau\in[0,T]$ and $i,j\in\{1,...,N\}$, such that $(\lambda_i-\lambda_j)(\tau)=0$. Then for all $t\geq \tau$, $(\lambda_i-\lambda_j)(t)=0$ and $(\xi_i-\xi_j)(t)=0$.
\end{prop}

\begin{proof}
Let $\tau\in[0,T]$ and $i,j\in\{1,...,N\}$, such that $(\lambda_i-\lambda_j)(\tau)=0$. Then for all $t\geq \tau$, 
\begin{equation}\label{difflambda0}
(\lambda_i-\lambda_j)(t)=-2 e^{t-\tau}\int_\tau^t e^{-(s-\tau)}(\xi_i-\xi_j)(s)ds.
\end{equation}
 Suppose for instance that $i<j$. According to Proposition \ref{prop_orderlambda2}, for all $t\in[0,T]$, $(\lambda_i-\lambda_j)(t)\geq 0$. Since we made Hypothesis \ref{hyp_orderxi}, the right-hand side of equation (\ref{difflambda0}) is nonpositive. This is only possible if both sides are equally zero. Hence, for all $t\geq \tau$, $(\lambda_i-\lambda_j)(t)=0$ and $(\xi_i-\xi_j)(t)=0$.
\end{proof}

The following proposition states that if at a certain point in time, two agents have the same projected velocities, then these should stay identical until final time. 

\begin{prop}\label{prop_equalcontrol}
Suppose that there exists $\tau\in [0,T]$ and $i,j \in \{1,...,N\}$ such that $\xi_i(\tau)=\xi_j(\tau)$. Then
\begin{equation}\label{equalcontrol}
 \text{for all } t\geq\tau, \; \xi_i(t)=\xi_j(t).
\end{equation}
As a consequence, for almost all $t\geq\tau$, $\alpha_i(t)=\alpha_j(t)$.
\end{prop}

\begin{proof}
Let $\tau\in [0,T]$ and $i,j\in \{1,...,N\}$.  
Define $\tilde\tau=\sup\{t\geq \tau \; | \; \xi_i(t)=\xi_j(t) \text{ for all } t\in [\tau, \tilde\tau]\}$. Notice from \eqref{xidot} that this implies that $\alpha_i(t)=\alpha_j(t)$ for almost every $t\in [\tau, \tilde{\tau}]$. Let us prove that $\tilde{\tau}=T$. \\
Suppose that $\tilde\tau<T$. Then there exists $\delta>0$ such that for all $t\in ]\tilde\tau, \tilde\tau+\delta]$, $\xi_i(t)\neq \xi_j(t)$. 
Define $\beta$ such that $\beta=\alpha$ on $[0,\tilde\tau]$ and
\[ 
\begin{cases}
 \beta_i=\beta_j=\frac12(\alpha_i+\alpha_j) \\
 \beta_k=\alpha_k \text{ for } k\not =i,\; k \not =j
\end{cases}  \qquad \text{ on } ]\tilde\tau,T ],
\]
and denote by $\xi^\beta$ the corresponding trajectory. 
Notice that $\sum_k\alpha_k\equiv \sum_k\beta_k$, so according to \eqref{xidot}, $\bar{\xi}\equiv\bar{\xi}^\beta$. 
This implies that $\xi_k=\xi_k^\beta$ for all $k\neq i,j$. Moreover, $\alpha_i+\alpha_j\equiv \beta_i+\beta_j$ so for all $t\in [\tilde\tau,T]$, $(\xi_i+\xi_j)(t) =(\xi_i^\beta+\xi_j^\beta)(t)$.  Furthermore, $\xi_i^\beta$ and $\xi_j^\beta$ satisfy the same differential equation on $[\tilde\tau, T]$ and $\xi_i^\beta(\tau)=\xi_j^\beta(\tau)$, so for all $t\in [\tilde\tau,T]$, $\xi_i^\beta(t)=\xi_j^\beta(t)=\frac12 (\xi_i+\xi_j)(t)$.
Define $\mathbb{V}_\alpha$ and $\mathbb{V}_{\beta}$ as the cost functions associated respectively with the controls $\alpha$ and $\beta$.
Then $\mathbb{V}_{\beta}=\mathbb{V}_\alpha \text{ on } [0,\tilde\tau]$.
On $]\tilde\tau,T]$, 
\[
\begin{split}
\mathbb{V}_{\alpha}-\mathbb{V}_{\beta} & = 
\sum_k(\xi_k)^2-\sum_k(\xi_k^{\beta})^2 
 = (\xi_i)^2 +(\xi_j)^2 - (\xi_i^{\beta})^2 - (\xi_j^{\beta})^2\\ 
 & =(\xi_i)^2 +(\xi_j)^2- 2(\frac12 (\xi_i+\xi_j) )^2 = (\xi_i - \xi_j)^2.
  \end{split} 
\]
Hence, for all $t\in ]\tilde{\tau},\tilde{\tau}+\delta]$, $\mathbb{V}_\alpha(t) > \mathbb{V}_\beta(t)$, and for all $t\in [\tilde{\tau}+\delta,T]$, $\mathbb{V}_\alpha(t) \geq \mathbb{V}_\beta(t)$. 
We get $\int_0^T\mathbb{V}_{\beta}<\int_0^T\mathbb{V}_\alpha$, which contradicts that $\alpha$ is an optimal control. In conclusion, $\tau=T$, which proves the proposition.
\end{proof}

\subsection{Optimal full-strength control}

We design an optimal control strategy for Problem \ref{prob_FS2}: 

\begin{theorem}\label{th_intstrategy}
Let $J(t)=\{i\in\{1,...,N\} \; | \; \xi_i(t)=\max_j\xi_j(t)\}$. The following control $\alpha$ is optimal for Problem \ref{prob_FS2}:
\begin{equation}\label{optcontstratint}
\begin{cases}
\forall i\in J(t), \; \alpha_i(t)=\frac{1}{|J(t)|} \\
\forall i\not\in J(t), \; \alpha_i(t)=0.
\end{cases}
\end{equation}
\end{theorem}

\begin{proof}
According to Pontryagin's maximum principle and the expression of the Hamiltonian (\ref{Hlambda_int}), the optimal control strategy solving Problem \ref{prob_FS2} requires to set
$\sum_{i\in I(t)} \alpha_i(t) =1$ and $\alpha_k(t)=0$ for $k \not\in I(t)$, where $I(t):=\{i \; | \; \lambda_i(t)=\max_j \lambda_j(t)\}$. Furthermore, according to Proposition \ref{prop_difflambda0}, if
$\lambda_i(\bar t)=\lambda_j(\bar t)$, then $\xi_i(t)=\xi_j(t)$ for all $t\geq \bar t$, and according to Proposition \ref{prop_equalcontrol}, $\alpha_i(t)=\alpha_j(t)$ for almost every $t\geq \bar t $.
Hence, the optimal strategy in fact requires to set, for almost every $t\in [0,T]$, 
\begin{equation}\label{optcontstratint2}
\begin{cases}
\forall i\in I(t), \; \alpha_i(t)=\frac{1}{|I(t)|} \\
\forall i\not\in I(t), \; \alpha_i(t)=0, 
\end{cases}
\end{equation}
where $|\cdot |$ denotes the cardinality of a set. 
Let us prove that $I(t)=J(t)$ for almost every $t$. 
Assume that $i\in I(t)$ and \eqref{optcontstratint2} holds true. According to Proposition \ref{prop_orderlambda2}, the covectors are ordered, so $\lambda_1(t)=\cdots=\lambda_i(t)$. From Proposition \ref{prop_difflambda0} and Hypothesis \ref{hyp_orderxi}, this implies $\xi_1(t)=\cdots =\xi_i(t)$, so $i\in J(t)$. 
Conversely, assume that $i\in J(t)$. Then from Hypothesis \ref{hyp_orderxi}, $\xi_1(t)=\cdots=\xi_i(t)$. According to Proposition \ref{prop_equalcontrol}, $\alpha_1(t)=\cdots=\alpha_i(t)$. Since $\alpha(t)$ verifies \eqref{optcontstratint2}, we deduce that $i\in I(t)$. 
Therefore, $I(t)=J(t)$ for almost every $t\in [0,T]$ and the optimal strategies \eqref{optcontstratint2} and \eqref{optcontstratint} are equivalent. 

\end{proof}

Notice that the control strategy in the case of integral cost minimization with full control (Problem \ref{prob_FS2}) is equivalent to the Instantaneous decrease strategy of Prop. \ref{prop_inst} (taking $M=1$). It is more restrictive than the optimal strategy minimizing the final value of the migration functional with full control (Problem \ref{prob_FS}) seen in Section \ref{Sec:gencase}. Indeed, this control strategy cannot be sparse. In order to minimize $\int_0^T\mathbb{V}(t)dt$, one has to split the control among more and more agents. However, any optimal control solving Problem \ref{prob_FS2} is also optimal for Problem \ref{prob_FS}.

\subsection{Optimal control in the general case}

After designing the optimal strategy for Problem \ref{prob_FS2}, we show that Problems \ref{prob_main2} and \ref{prob_FS2} are actually equivalent, i.e. that the optimal control solving Problem \ref{prob_main2} belongs to $\mathcal{U}_{\text{FS}}$.

\begin{theorem}
The optimal control strategy for Problem \ref{prob_main2} requires using full-strength control, i.e. $\alpha\in\mathcal{U}_{\text{FS}}$. 
\end{theorem}

\begin{proof}
According to the Pontryagin Maximum Principle (see Section \ref{Sec:PMPint}), if $\lambda_1(t)>0$ for all $t$, then full control must be used at all time.  
Combining the final condition \eqref{lambda_final2} and the evolution \eqref{Hlambda_int}, we get $\lambda_1(T)=0$ and $\dot{\lambda}_1(T)=-2\xi_1(T)<0$. Hence there exists an interval $]t,T[$ on which $\lambda_1>0$. Let $\tau=\inf\{t\in [0, T] \text{ s.t. } \lambda_1(s)>0 \text{ for all } s\in ]t,T[ \}$. Suppose that $\tau>0$. Then $\lambda_1(\tau)=0$. Furthermore, $\dot\lambda_1(\tau)=(\lambda_1-\bar\lambda)(\tau) -2\xi_1(\tau)$. We compute: $\dot\lambda_1-\dot{\bar\lambda}=\lambda_1-\bar\lambda -2(\xi_1-\bar\xi)$. Denoting $\Lambda=\lambda_1-\bar{\lambda}$, we get the following evolution backwards in time: $\dot{\Lambda}=-\Lambda+2(\xi_1-\bar\xi)$. Recall that backwards in time, we also have: $\dot\xi_1=\xi_1-(1-\alpha_1)\bar\xi$. If $\Lambda=\xi_1$, then $\dot\Lambda= \xi_1-2\bar\xi=\dot\xi_1+(1-\alpha_1)\bar\xi-2\bar\xi=\dot\xi_1-(1+\alpha_1)\bar\xi<\dot\xi_1$. Since $\Lambda(T)=0<\xi_1(T)$, this implies that $\Lambda(t)<\xi_1(t)$ for all $t\in[\tau,T]$. Hence, $\dot\lambda_1(\tau) = \Lambda(\tau)-2\xi_1(\tau) <0$, which contradicts the definition of $\tau$. We conclude that $\lambda_1(t)>0$ for all $t\in]0,T[$, and that $\sum_i\alpha_i\equiv 1$.
\end{proof}

Hence, the control strategy designed in Theorem \ref{th_intstrategy} is an optimal strategy for the minimization of integral cost (Prob. \ref{prob_main2}). Unlike in the minimization of the final cost (Prob. \ref{prob_main}), there is no initial Inactivation period. 

Figure \ref{fig:t} illustrates the control strategy designed in Theorem \ref{th_intstrategy}. In this example, 5 agents are to be controlled optimally to reach consensus at the target velocity $V=(1,0)$. Initially (Figure \ref{fig:t1}), only one agent is controlled, the agent with the biggest projected velocity over $\bar{v}-V$. The set $J(t)=\arg\max_{i\in\{1,...,N\} } \langle v_i, \frac{\bar{v}-V}{\|\bar{v}-V\|}\rangle $ contains more and more agents as time goes (\ref{fig:t2}, \ref{fig:t3}) and eventually, control is split evenly among all agents (see Figure \ref{fig:t4}).  

\begin{figure}[h]
        \centering
        \begin{subfigure}[b]{0.22\textwidth}
                \includegraphics[trim=1cm 0cm 2cm 1cm, clip=true, scale=0.25]{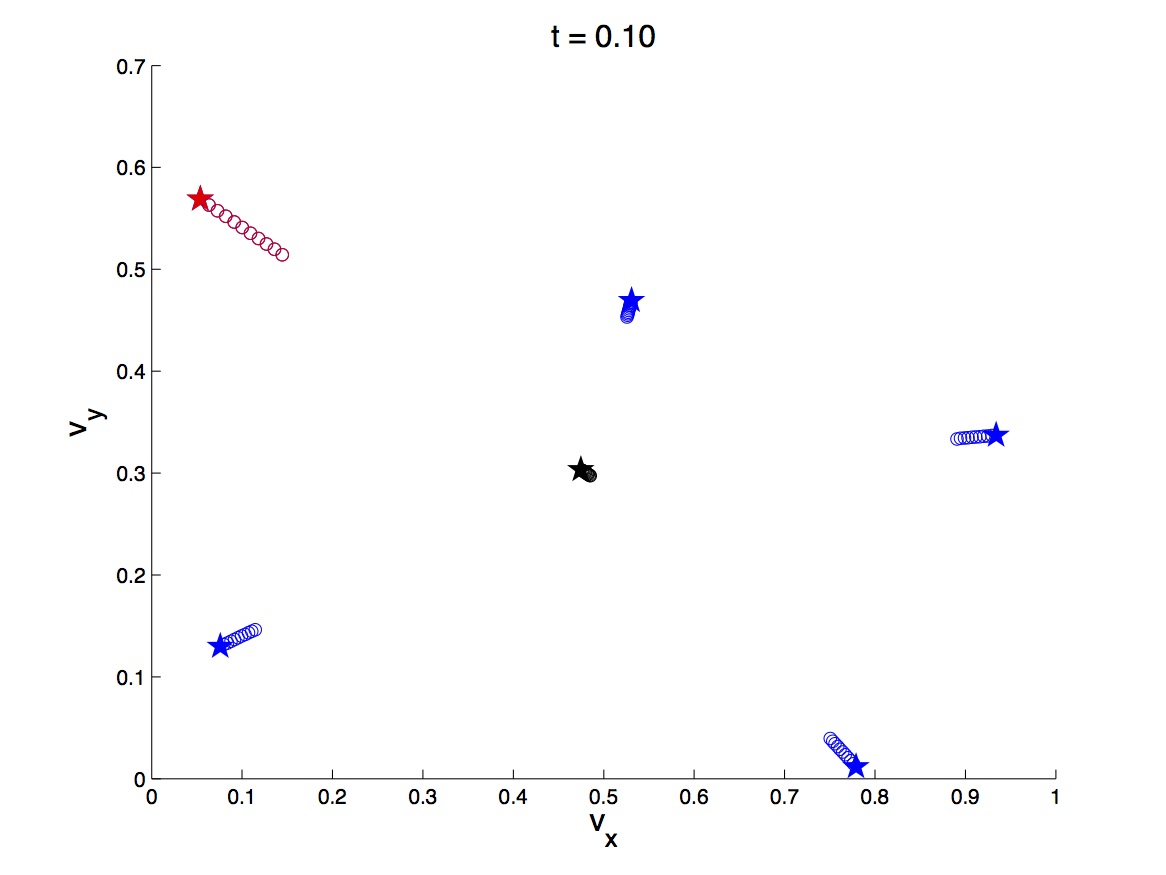}
                \caption{$t=0.1$}
                \label{fig:t1}
        \end{subfigure}%
        \begin{subfigure}[b]{0.22\textwidth}
                \includegraphics[trim=1cm 0cm 2cm 1cm, clip=true, scale=0.25]{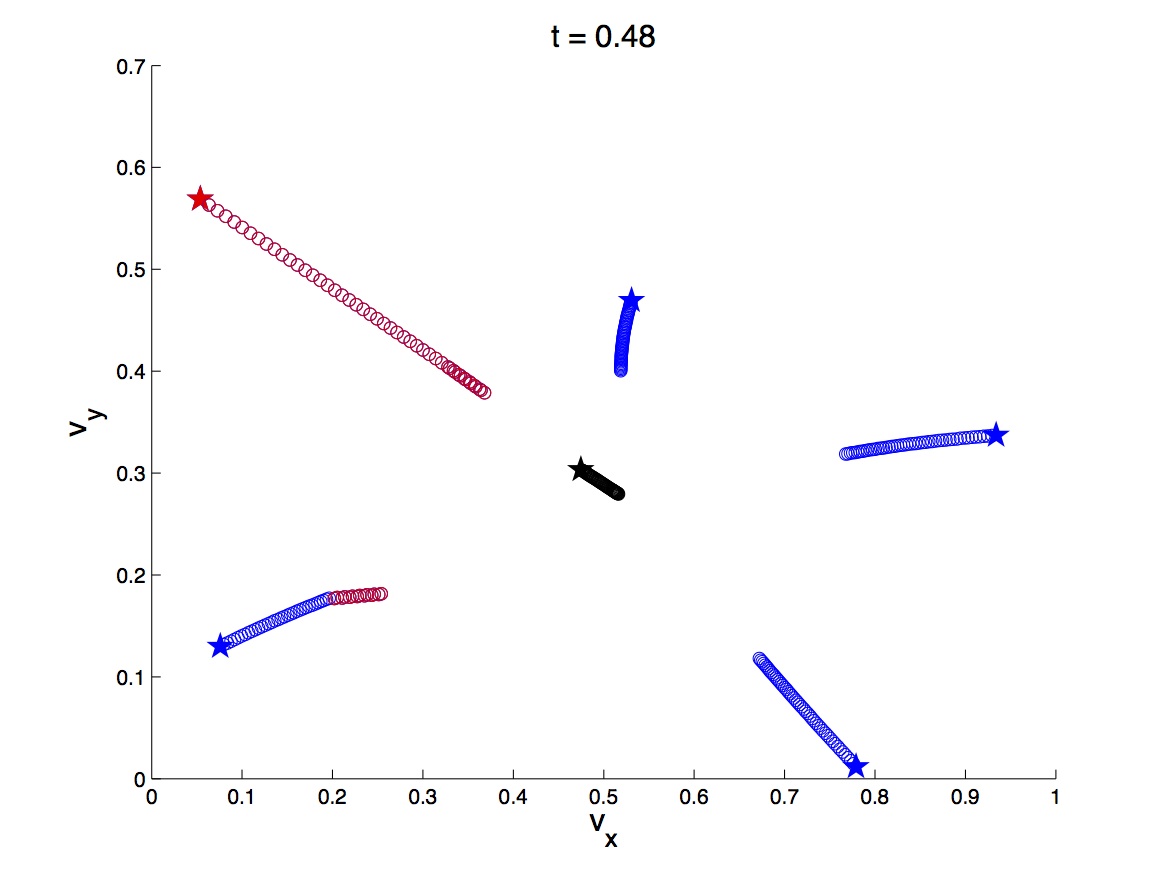}
                \caption{$t=0.48$}
                \label{fig:t2}
        \end{subfigure}
        \begin{subfigure}[b]{0.22\textwidth}
                \includegraphics[trim=1cm 0cm 2cm 1cm, clip=true, scale=0.25]{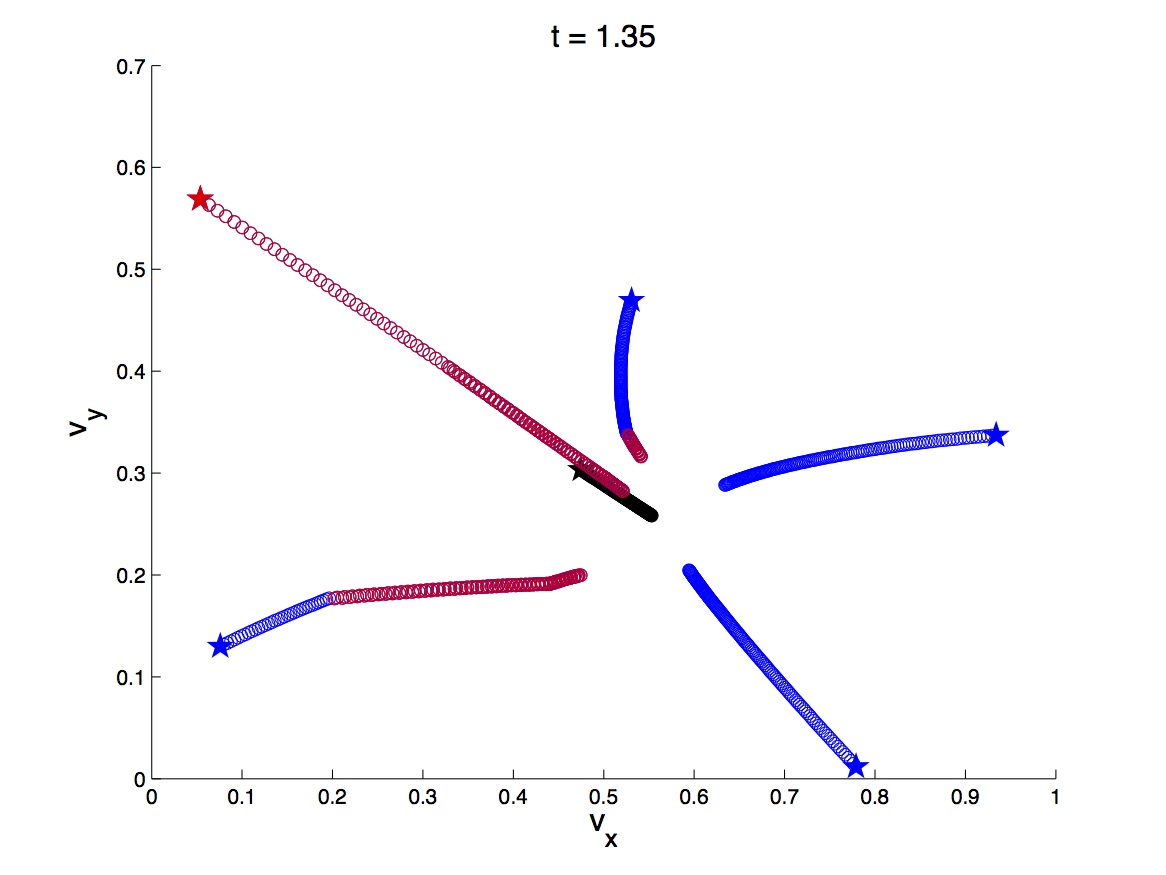}
                \caption{$t=1.35$}
                \label{fig:t3}
        \end{subfigure}
         \begin{subfigure}[b]{0.22\textwidth}
                \includegraphics[trim=1cm 0cm 2cm 1cm, clip=true, scale=0.25]{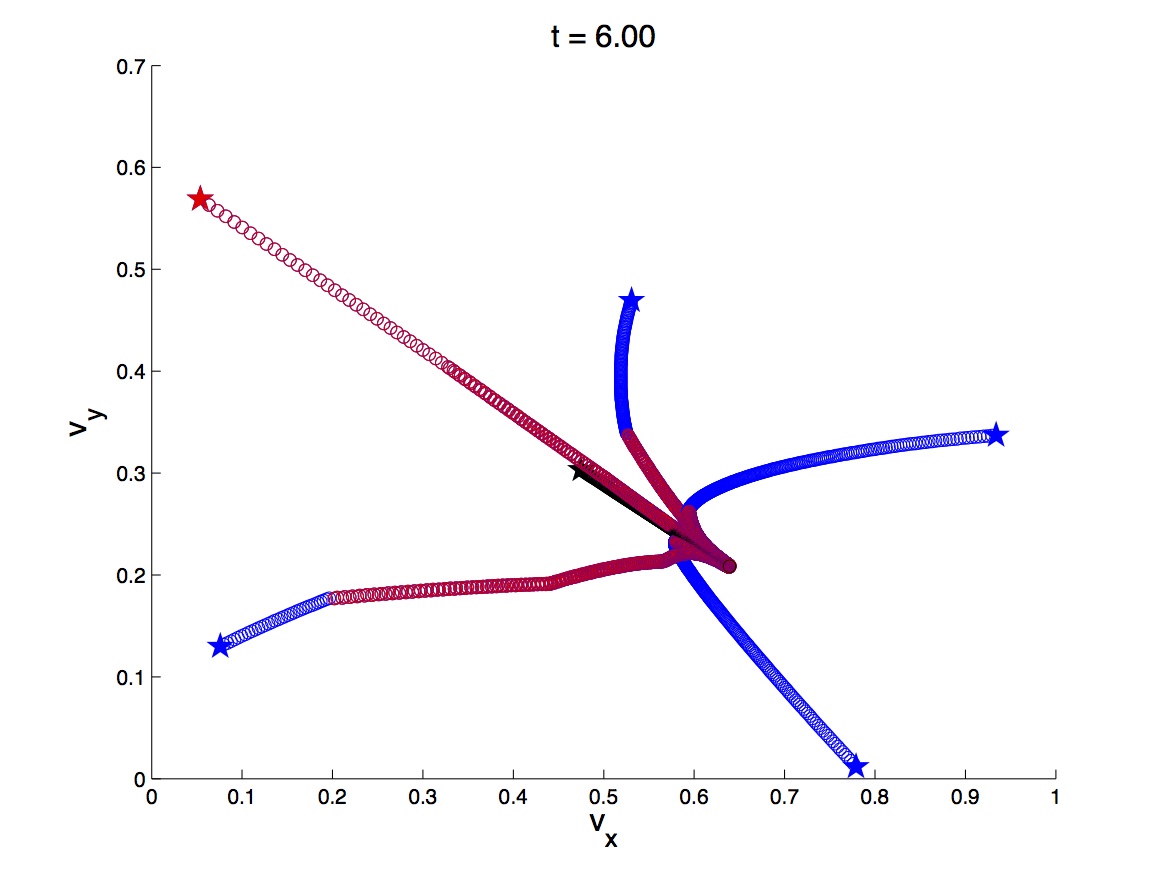}
                \caption{$t=6.0$}
                \label{fig:t4}
        \end{subfigure}
        \caption{Control of 5 agents to reach the target velocity $V=(1,0)$. Agents are represented in the velocity space, controlled ones in red, uncontrolled ones in blue, and the mean velocity in black. Initial positions are marked by stars.}\label{fig:t}
\end{figure}

\section*{Acknowledgment}
The authors acknowledge the support of the NSF project KI-Net, DMS Grant \# 1107444. 
Benjamin Scharf acknowledges the support of the ERC-Starting Grant "High-Dimensional Sparse Optimal Control".

\section*{Appendix}

\begin{proof}(Theorem \ref{Th_M12})
First, let $\xi_1(0)\geq\xi_2(0) > 0$. According to Prop. \ref{prop_positive}, $\xi_1(t)> 0$ and $\xi_2(t) > 0$ for all $t\in [0,T]$. The transversality condition gives $\lambda_1(T)> 0$ and $\lambda_2(T)> 0$, and according to Prop. \ref{prop_M12}, $\lambda_1(t)> 0$ and $\lambda_2(t)> 0$ for all $t\leq T$. According to Pontryagin's maximum principle (see Section \ref{Sec:PMP}), the global strategy requires setting $\alpha_1+\alpha_2\equiv M$. In this case, $\bar\xi(t) = \bar\xi(0) \exp(-\frac{M}{2}t)$ does not depend on the choice of $\alpha_1$ and $\alpha_2$. Minimizing $\mathbb{V}$ (\ref{functional_V2}) therefore amounts to minimizing $(\xi_1-\xi_2)^2$. 
\begin{itemize}\itemsep0pt \parskip0pt \parsep0pt
\item[(i)] If $T\geq t_2$, we will show that in addition to satisfying $\alpha_1+\alpha_2\equiv M$, the optimal control $\alpha$ must achieve $\xi_1(T)=\xi_2(T)$. Such a control strategy exists, since for instance (as one can see by direct computation of \eqref{scalar2}) the control $(\beta_1,\beta_2)(t)=(1,M-1)$ for all $t\in [0,t_2[$ and $(\beta_1,\beta_2)(t)=(M/2,M/2)$ for all $t\in [t_2,T]$ achieves $\xi_1^\beta(t)=\xi_2^\beta(t)$ for all $t\in [t_2,T]$, where $\xi^\beta$ denotes the corresponding trajectory. Notice that $\beta$ minimizes $\bar{\xi}(T)$ by using the full strength $M$ of the control at all time (see \eqref{xibar}), and minimizes $(\xi_1-\xi_2)^2(T)$, so it minimizes $\mathbb{V}(T)$ (see \eqref{functional_V}). Hence, in order to be optimal, $\alpha$ must satisfy $\xi_1(T)=\xi_2(T)$ as well as $\alpha_1+\alpha_2\equiv M$.
\item[(ii)] If $T<t_2$, we will show that $(\alpha_1,\alpha_2)\equiv (1,M-1)$ and that $\xi_1$ and $\xi_2$ cannot be brought together (i.e. $\xi_1(T)>\xi_2(T)$).
Indeed, knowing that $\alpha_1+\alpha_2\equiv M$, one can use \eqref{diffxi1} to compute: $(\xi_1-\xi_2)(t)=e^{-t}\left( (\xi_1-\xi_2)(0)-\int_0^t (\alpha_1-\alpha_2)(s)\bar{\xi}(s)e^s ds \right)$.
 Since $\bar{\xi}$ is fully determined, $t_{\text{min}}:=\min_{\alpha\in\mathcal{U}_M, \alpha_1+\alpha_2\equiv M} \{t\in [0,T] \text{ s.t. } (\xi_1-\xi_2)(t)=0\}$ is achieved by maximizing $(\alpha_1-\alpha_2)$, which gives: $(\alpha_1, \alpha_2)\equiv (1,M-1)$. As seen previously, by direct computation of \eqref{scalar2}, $t_{\text{min}}=t_2$ as defined above. Hence, if $T<t_2$, necessarily $\xi_1(T)>\xi_2(T)$. 
Then $\lambda_1(T)>\lambda_2(T)$ and according to Prop. \ref{prop_M12}, and to Prop. \ref{prop_equality}, $\lambda_1(t)>\lambda_2(t)> 0$ for all $t$. According to the PMP (see \ref{Sec:PMP}), the optimal strategy is $(\alpha_1, \alpha_2)\equiv (1,M-1)$. 
\end{itemize}
Now let $\xi_1(0)> 0, \; \xi_2(0)< 0$ and $\bar{\xi}(0)>0$. 
 We then distinguish four subcases.\\
 Firstly, let us prove that if $\xi_2(T)> 0$, then necessarily $T> t_1$. Indeed, if $0<\xi_2(T)\leq \xi_1(T)$, then $0<\lambda_2(T)\leq \lambda_1(T)$, and according to Proposition \ref{prop_M12}, $0<\lambda_2(t)\leq \lambda_1(t)$ for all $t\in [0,T]$. According to the PMP (see Section \ref{Sec:PMP}), $\alpha_1+\alpha_2\equiv M$. Hence $\bar{\xi}(t)=\bar{\xi}(0)e^{-Mt/2}$, and $\xi_2(t)=e^{-t}(\xi_2(0)+\bar{\xi}(0)\int_0^t(1-\alpha_2)e^{\frac{2-M}{2}s}ds)$. The minimum time $t_\text{min}$ needed to achieve $\xi_2(t_\text{min})>0$ is achieved for $(\alpha_1,\alpha_2)\equiv(1,M-1)$, which, after computation, gives $t_\text{min}=t_1$ as defined above. Hence, if $\xi_2(T)\geq 0$, then $T> t_1$.
\begin{itemize}\itemsep0pt \parskip0pt \parsep0pt
\item[(iii)] Let $T\geq t_2$. Let us prove that $\alpha_1+\alpha_2\equiv M$ and $\xi_1(T)=\xi_2(T)$. Such a control strategy exists. Indeed, take for example $(\beta_1,\beta_2)(t)=(1,M-1)$ on $[0,t_2]$ and $(\beta_1,\beta_2)(t)=(M/2,M/2)$ on $]t_2,T]$. Then, by direct computation of \eqref{scalar2}, $\xi_1^\beta(t)=\xi_2^\beta(t)$ for all $t\in [t_2,T]$ (where $\xi^\beta$ denotes the trajectory corresponding to the control $\beta$). Furthermore, $\beta$ is optimal since it minimizes $\bar{\xi^\beta}$ by using full control at all time and achieves $(\xi_1^\beta-\xi_2^\beta)^2(T)=0$ (see \eqref{functional_V}). In order to perform optimally, the control $\alpha$ must also satisfy $\alpha_1+\alpha_2\equiv M$ and $\xi_1(T)=\xi_2(T)$.
\item[(iv)] Let $T<t_0$. Since $t_0<t_1$, then as proved above, $\xi_2(T)\leq 0$. Suppose that $\xi_2(T)=0$. Then $\lambda_1(T)>\lambda_2(T)=0$ and according to Proposition \ref{prop_M12}, $\lambda_1(t)>\lambda_2(t)=0$ for all time $t$. Hence, $
\alpha_1\equiv 1$ (see Section \ref{Sec:PMP}). Then $\min_{\alpha_2} \{t\in [0,T] \text{ s.t. } \xi_2(t)=0\} = t_0$ as defined above (obtained for $\alpha_2\equiv 0$). This contradicts the condition on $T$. Hence, if $T<t_0$, then $\xi_2(T)<0$ and according to Proposition \ref{prop_M12} and Section \ref{Sec:PMP}, $\lambda_2<0$ so $ \alpha_2\equiv 0$. However, there is no information on $\lambda_1$ other than $\dot{\lambda}_1=\alpha_1 /2 \lambda_1+\bar{\lambda}-\lambda_1\geq 0$ and $\lambda_1(\tau)=0$ implies $\dot{\lambda}_1(\tau)>0$. Hence, as in the previous sections, there exists $t^*\in [0,T[$ such that $\lambda_1<0$ on $[0,t^*[$, $\lambda_1(t^*)=0$ and $\lambda_1>0$ on $]t^*,T]$. This implies that $(\alpha_1,\alpha_2)=(0,0)$ on $[0,t^*[$ and $(\alpha_1,\alpha_2)=(1,0)$ on $[t^*,T[$.
\item[(v)] Let $t_0\leq T \leq t_1$. We shall prove that $\xi_2(T)=0$ and that $\alpha_1\equiv 1$. As seen previously, if $T\leq t_1$, then $\xi_2(T)\leq 0$. Suppose that $\xi_2(T)<0$. Then $\lambda_1(T)>0$ and $\lambda_2(T)<0$ which according to Proposition \ref{prop_M12} gives $\lambda_2(t)<0$ for all $t$, and according to the PMP (see Section \ref{Sec:PMP}), $\alpha_2\equiv 0$. Then $\xi_2(t)=e^{-t} (\xi_2(0)+\bar{\xi}(0)\int_0^T e^{-\int_0^s\frac12\alpha_1(r)dr}e^s ds)$. Thus $t_{\text{sup}}:=\sup_{\alpha_1} \{\tau\in [0,T] \text{ s.t. } \xi_2(t)<0 \text{ for all } t\in [0,\tau[ \}$ is obtained for $\alpha_1\equiv 1$ and by direct computation, $t_{\text{sup}}=t_0$. Since $T\geq t_0$, there exists $\tau\leq T$ such that $\xi_2(\tau)=0$. However, by Proposition \ref{prop_positive}, once $\xi_2=0$ it cannot become negative again, which contradicts $\xi(T)<0$. Therefore, $\xi_2(T)=0$, and according to Proposition \ref{prop_M12} and the PMP (Section \ref{Sec:PMP}), $\lambda_1(t)>0$ for all $t\in [0,T]$ so $\alpha_1\equiv 1$. Furthermore, if $\xi_2(\tau)=0$, then $\dot{\xi}_2(\tau)=(1-\alpha_2(\tau))\bar{\xi}(\tau)>0$ since $\alpha_2=M-\alpha_1=M-1<1$. According to Proposition \ref{prop_positive}, once $\xi_2$ becomes positive it cannot become zero again. Hence we must have $\xi_2(t)<0$ for all $t<T$ and $\xi_2(T)=0$.
\item[(vi)] Let $t_1<T<t_2$. 
As in the previous case, since $T\geq t_0$, one must have: $\xi_2(T)\geq 0$. Suppose that $\xi_2(T)=0$. Then according to Proposition \ref{prop_M12} and the PMP, $\alpha_1\equiv 1$ and $\xi_2(t)=e^{-t} (\xi_2(0)+\bar{\xi}(0)\int_0^T (1-\alpha_2)(s) e^{-\int_0^s\frac12(1+\alpha_2)(r)dr}e^s ds)$. Then the minimum of $\xi_2(T)$ is obtained for $\alpha_2\equiv M-1$, so
\vspace{-0.3cm}
\[
\xi_2(T)\geq e^{-T} (\xi_2(0)+\bar{\xi}(0)\int_0^T (2-M) e^{-\frac12 M s}e^s ds) > e^{-T} (\xi_2(0)+\bar{\xi}(0) ( e^{\frac{2-M}{2}t_1}-1) ) > 0
\vspace{-0.2cm}
\]
by definition of $t_1$. This contradicts $\xi_2(T)=0$, so necessarily $\xi_2(T)>0$. Then $\lambda_1(t)>0$ and $\lambda_2(t)>0$ for all $t$, which implies that $\alpha_1+\alpha_2\equiv M$. In this case we prove as in case {\it (ii)} that $\xi_1(T)>\xi_2(T)$, which implies $(\alpha_1,\alpha_2)\equiv (1,M-1)$.
\end{itemize}
\vspace{-0.6cm}
\end{proof}


\newpage

\begin{thebibliography}{9}

\bibitem{Bae} H.-O. Bae, S.-Y. Ha, Y. Kim, S.-H. Lee, H.
Lim, J.Yoo, Mathematical model for volatility flocking with a regime switching mechanism in a stock market, \textit{Math. Models Methods Appl. Sci.}, \textbf{25} (2015), 12991335.

\bibitem{Bellomo} N. Bellomo and A. Bellouquid, On the modeling of crowd  dynamics: Looking at the beautiful shapes of swarms,\textit{ Netw. Heter. Media} \textbf{6} (2011) 383--399.


\bibitem{Bellomo2} N. Bellomo and J. Soler, On the mathematical theory of the dynamics of swarms viewed as complex systems, \textit{Math. Models Methods Appl. Sci. } \textbf{22} (2012) 1140006.

\bibitem{Berman} S. Berman, Q. Lindsey, M. S. Sakar, V. Kumar, and
S. C. Pratt, Experimental study and modeling of group
retrieval in ants as an approach to collective transport
in swarm robotic systems, \textit{Proceedings of the IEEE, 99}
\textbf{9} (2011) 1470--1481.

\bibitem{BressPic}
A. Bressan and B. Piccoli, \textit{Introduction to the Mathematical Theory of Control} (AIMS on Applied Math, Vol.2, 2007).

\bibitem{Camazine} S. Camazine, J. Deneubourg, N. Franks, J. Sneyd, G. Theraulaz, and E. Bonabeau. \textit{Self organization in biological systems} (Princeton University Press, 2003).

\bibitem{Caponigro2}
M. Caponigro, M. Fornasier, B. Piccoli, E. Tr\'elat, Sparse stabilization and control of alignment models, \textit{Math. Models Methods Appl. Sci.},  \textbf{25}(03) (2015), 521-564. 

\bibitem{Caponigro}
M. Caponigro, M. Fornasier, B. Piccoli, E. Tr\'elat, Sparse stabilization and optimal control of the Cucker-Smale model, \textit{Math. Cont. Related Fields} \textbf{3} (2013) 447--466.
     
\bibitem{Chuang}
     Y. Chuang, Y. Huang, M. D'Orsogna, and A. Bertozzi, Multi-vehicle flocking: scalability of cooperative
control algorithms using pairwise potentials, \textit{IEEE International Conference on Robotics and Automation} (2007) 2292--2299.
     
\bibitem{Couzin}
	I. Couzin and N. Franks, Self-organized lane formation and optimized traffic flow in army ants,
\textit{Proc. R. Soc. Lond. B} \textbf{270} (2002) 139--146.

\bibitem{CouzinKrause}
	I. Couzin, J. Krause, N. Franks, and S. Levin, Effective leadership and decision making in animal groups on the move, \textit{Nature} \textbf{433} (2005) 513--516.
	
\bibitem{Cristiani} E. Cristiani, B. Piccoli, and A. Tosin, Modeling self-organization in pedestrians and animal
groups from macroscopic and microscopic viewpoints, in G. Naldi, L. Pareschi, G. Toscani, and N. Bellomo, editors, \textit{Mathematical Modeling of Collective Behavior in Socio-Economic and Life
Sciences}, Modeling and Simulation in Science, Engineering and Technology (Birkh\"auser Boston, 2010) 337--364.	

\bibitem{Cristiani2} E. Cristiani, B. Piccoli, and A. Tosin, Multiscale modeling of granular flows with application to
crowd dynamics, \textit{Multiscale Model. Simul.} \textbf{9} (2011) 155--182.
	
\bibitem{CuckerSmale}
F. Cucker and S. Smale, Emergent behavior in flocks, \textit{IEEE Trans.
Automat. Control} \textbf{52} (2007) 852--862.	

\bibitem{Dall} 
S. R. X. Dall, L.-A. Giraldeau, O. Olsson, J. M. McNamara,
and D. W. Stephens. Information and its use by animals in evolutionary ecology, \textit{Trends in Ecology \&
Evolution} \textbf{20} (2005) 187--193.

\bibitem{During}
B. D\" uring, D. Matthes, and G. Toscani, Kinetic equations modelling wealth redistribution: A
comparison of approaches, \textit{Phys. Rev. E} \textbf{78} (2008) 056103.

\bibitem{Fornasier} 
M. Fornasier, B. Piccoli, and F. Rossi, Mean-field sparse optimal control, \textit{Phil. Trans. R. Soc. A.} \textbf{372} (2014) 20130400.

\bibitem{Gauthier} J.P. Gauthier, The Inactivation principle: Mathematical solutions minimizing the absolute work and biological implications for the planning of arm movements, \textit{PLoS Comput. Biol.} \textbf{4}(10) (2008).

\bibitem{Guttal} V. Guttal and I. D. Couzin, Social interactions, information
use, and the evolution of collective migration,
\textit{Proceedings of the National Academy of Sciences} \textbf{107}(37) (2010) 16172--16177.

\bibitem{Tadmor} S.-Y. Ha and E. Tadmor, From particle to kinetic and hydrodynamic descriptions of flocking, \textit{Kinet. Relat. Models} \textbf{1}  (2008) 415--435.

\bibitem{Horstmann} D. Horstmann, From 1970 until present: The Keller-Segel model in chemotaxis and its consequences,
\textit{I. Jahresber. Dtsch. Math.-Ver.} \textbf{105}(3) (2003) 103--165.

\bibitem{Horstmann2} D. Horstmann, From 1970 until present: the Keller-Segel model in chemotaxis and its consequences,
II. \textit{Jahresber. Dtsch. Math.-Ver.} \textbf{106} (2004) 51--69.

\bibitem{Jadbabaie}
A. Jadbabaie, J. Lin, and A. S. Morse, Correction to: "Coordination of groups of mobile autonomous
agents using nearest neighbor rules" [IEEE Trans. Automat. Control 48, no. 6, 2003.
988--1001; MR 1986266], \textit{IEEE Trans. Automat. Control} \textbf{48}(9) (2003) 1675.	

\bibitem{Keller} E. F. Keller and L. A. Segel, Initiation of slime mold aggregation viewed as an instability, \textit{J.
Theor. Biol.} \textbf{26}(3) (1970) 399--415.

\bibitem{Lasry} J.-M. Lasry and P.-L. Lions, Mean field games, \textit{Jpn. J. Math} \textbf{2}(1) (2007) 229--260.

\bibitem{Leonard1}N. Leonard, Multi-Agent System Dynamics: Bifurcation
and Behavior of Animal Groups, \textit{Proc. 9th IFAC Symposium on Nonlinear Control Systems} 307--317.

\bibitem{Leonard}
N. Leonard and E. Fiorelli, Virtual leaders, artificial potentials and coordinated control of groups,
\textit{Proc. 40th IEEE Conf. Decision Contr.} (2001) 2968--2973.

\bibitem{Motsch}
S. Motsch, E. Tadmor, How heterophilious dynamics enhance consensus, \textit{SIAM review} \textbf{56}(4) (2014) 577--621.

\bibitem{Niwa}
	H. Niwa, Self-organizing dynamic model of fish schooling, \textit{J. Theor. Biol.} \textbf{171} (1994) 123--136.

\bibitem{Parrish}
J. Parrish and L. Edelstein-Keshet, Complexity, pattern, and evolutionary trade-offs in animal aggregation, \textit{Science} \textbf{294} (1999) 99--101.

\bibitem{Parrish2}
J. Parrish, S. Viscido, and D. Gruenbaum. Self-organized fish schools: An examination of emergent properties, \textit{Biol. Bull.} \textbf{202} (2002) 296--305.
	
\bibitem{Patlak}
C. S. Patlak, Random walk with persistence and external bias, \textit{Bull. Math. Biophys.} \textbf{15} (1953) 311--338.

\bibitem{Perea}
L. Perea, G. G´omez, and P. Elosegui, Extension of the Cucker-Smale control law to space flight
formations, \textit{AIAA Journal of Guidance, Control, and Dynamics}, \textbf{32} (2009) 527--537.

\bibitem{Perthame}
B. Perthame, \textit{Transport Equations in Biology} (Basel: Birkhauser, 2007).

\bibitem{Pont}
L.S. Pontryagin, V.G.Boltyanskii, R.V. Gamkrelidze and E.F. Mishenko, \emph{The Mathematical Theory of Optimal Processes} (John Wiley and Sons, New York, 1962).

\bibitem{Romey}
W. Romey, Individual differences make a difference in the trajectories of simulated schools of fish, \textit{Ecol. Model.} \textbf{92} (1996) 65--77. 

\bibitem{Sepulchre}
R. Sepulchre, D. Paley, and N. E. Leonard, Stabilization of planar collective motion with all-to-all
communication, \textit{IEEE Transactions on Automatic Control} \textbf{52}(5) (2007) 811--824. 

\bibitem{Sugawara}
K. Sugawara and M. Sano, Cooperative acceleration of task performance: Foraging behavior of interacting multi-robots system, \textit{Physica D} \textbf{100} (1997) 343--354.

\bibitem{Toner}
J. Toner and Y. Tu, Long-range order in a two-dimensional dynamical xy model: How birds fly together, \textit{Phys. Rev. Lett.} \textbf{75} (1995) 4326--4329.

\bibitem{Viscek}
T. Vicsek, A. Czirok, E. Ben-Jacob, I. Cohen, and O. Shochet, Novel type of phase transition in a system of self-driven particles, \textit{Phys. Rev. Lett.}  \textbf{75} (1995) 1226--1229.


     \end{thebibliography}
\end{document}